\documentclass[a4paper,11pt]{amsart}
  \usepackage{amsmath}
  \usepackage{url,graphicx}
  \usepackage{amssymb, color, pstricks-add, manfnt}
  \usepackage{amsmath, amsthm,  amsfonts}
  \usepackage{mathrsfs}

  \theoremstyle{plain}
  \newtheorem{Theorem}{Theorem}[section]
  \newtheorem{Lemma}{Lemma}[section]

  \newtheorem{Corollary}{Corollary}[section]

    \theoremstyle{remark}
  \newtheorem{remark}{Remark}

  \numberwithin{equation}{section}
  \numberwithin{figure}{section}
  \numberwithin{remark}{section}

\renewcommand{\baselinestretch}{1.00}
\begin{document}

\title{Oblique boundary value problems for augmented Hessian equations II}

\author{Feida Jiang}
\address{Yau Mathematical Sciences Center, Tsinghua University, Beijing 100084, P.R. China}
\email{jiangfeida@math.tsinghua.edu.cn}

\author{Neil S. Trudinger}
\address{Centre for Mathematics and Its Applications, The Australian National University, Canberra ACT 0200, Australia}
\email{Neil.Trudinger@anu.edu.au}

\thanks{Research supported by  National Natural Science Foundation of China (No.11401306),  Australian Research Council (No.DP1094303), China Postdoctoral Science Foundation (No.2015M571010) and Jiangsu Natural Science Foundation of China (No.BK20140126).}


\date{\today}


\abstract {In this paper, we continue our investigations into the global theory of  oblique boundary value problems for augmented Hessian equations. We construct a global barrier function in terms of an admissible function in a uniform way when the matrix function in the augmented Hessian is only assumed regular. This enables us to derive global second derivative estimates in terms of boundary estimates which are then obtained by strengthening the concavity or monotonicity conditions in our previous work on the strictly regular case. Finally we give some  applications to existence theorems which embrace standard Hessian equations as special cases.}

\endabstract

\maketitle


\baselineskip=12.8pt
\parskip=3pt
\renewcommand{\baselinestretch}{1.38}

\section{Introduction}\label{Section 1}
\vskip10pt

This paper is a continuation of our previous paper \cite{JT-oblique-I} on  augmented Hessian partial differential equations of the form
\begin{equation}\label{1.1}
\mathcal{F}[u]:=F[D^2u-A(\cdot,u,Du)]=B(\cdot,u,Du), \quad {\rm in} \ \Omega,
\end{equation}
together with oblique boundary conditions
\begin{equation}\label{1.2}
\mathcal{G}[u]:=G(\cdot,u,Du)=0, \quad {\rm on} \ \partial\Omega.
\end{equation}
Here $\Omega\subset \mathbb{R}^n$ is a  smooth bounded domain, $Du$ and $D^2u$ denote the gradient vector and the Hessian matrix respectively of the function $u\in C^2(\Omega)$ respectively, $A$ is a $n\times n$ symmetric matrix function on $\Omega\times \mathbb{R}\times \mathbb{R}^n$, $B$ and $G$ are scalar valued functions on $\Omega\times \mathbb{R}\times \mathbb{R}^n$ and $\partial\Omega\times \mathbb{R}\times \mathbb{R}^n$ respectively and $F$ is a scalar valued function on $\mathbb{S}^n$, the linear space of $n\times n$ symmetric matrices. We shall denote the points in $\Omega, \mathbb{R}, \mathbb{R}^n$ and $\mathbb{S}^n$ by $x, z, p$ and $r$, respectively. Assuming $G$ is  differentiable with respect to $p$, the boundary condition \eqref{1.2} is  oblique with respect to $u$ if
\begin{equation}\label{obliqueness}
G_p(\cdot,u,Du)\cdot \nu \ge \beta_0, \quad {\rm on}\ \partial\Omega,
\end{equation}
where $\nu$ is the unit inner normal vector field on $\partial\Omega$ and $\beta_0$ is a positive constant.
We simply call $G$ (or $\mathcal{G}$) oblique, if $G_p\cdot \nu>0$ on $\partial\Omega\times \mathbb{R}\times \mathbb{R}^n$. For the functions $F, G, A$ and $B$, we will make appropriate smoothness assumptions as necessary in the context.

Comparing with the standard Hessian equations, where $A=0$, the key ingredients for the regularity of solutions to the augmented Hessian equations are certain convexity conditions on the matrix $A$ with respect to the $p$ variable.  The Heinz-Lewy counterexample \cite{Schulz1990} shows that there is no $C^1$ regularity even for the Monge-Amp\`ere case in two dimensions in equation \eqref{1.1} without suitable structures on $A$. Assuming  $A$ is twice differentiable with respect to  $p$, we call the matrix $A$ co-dimensional one convex (strictly co-dimensional one convex) with respect to $p$, if
\begin{equation}\label{MTW}
\sum_{i,j,k,l}A^{kl}_{ij}(x,z,p)\xi_i\xi_j\eta_k\eta_l \ge 0, \ (>0),
\end{equation}
for all $(x,z,p)\in \Omega \times \mathbb{R}\times \mathbb{R}^n$, $\xi,\eta\in \mathbb{R}^n$, $\xi\cdot\eta=0$, where $A^{kl}_{ij}:=D^2_{p_kp_l}A_{ij}$.
For simplicity, we call the matrix $A$ regular (strictly regular) if $A$ is co-dimensional one convex (strictly co-dimensional one convex) with respect to $p$. In \cite{JT-oblique-I}, we have established  second derivatives estimates and resultant classical solvability results for the oblique problem \eqref{1.1}-\eqref{1.2} when the matrix $A$ is strictly regular. In this paper we consider the case when $A$ is only assumed regular. Our results depend on offsetting the lack of strict regularity of $A$ by either uniform concavity conditions on $G$ with respect to $p$ or strong monotonicity  conditions on $G$ or $A$ with respect to $z$.

The organization of this paper is as follows. In Section \ref{Section 2}, extending our previous barrier constructions in \cite{JTY2013, JT2014, JTY2014}, we construct a barrier function in Lemma \ref{Lemma 2.1} from an admissible function in a uniform way when $A$ is regular. Such a construction is already indicated in \cite {JTX2015}. In Section \ref{Section 3}, we deal with the global second derivative estimate and the second order derivative estimate on the boundary in Section \ref{Subsection 3.1} and Section \ref{Subsection 3.2} respectively. In Section \ref{Subsection 3.1}, using the constructed barrier in Lemma \ref{Lemma 2.1}, we reduce the second derivative estimate to the boundary in Theorem \ref{Th3.1}, following the special cases of Monge-Ampere operators in \cite{JTY2013} and $k$-Hessian operators in \cite{JTY2014}. In Section \ref{Subsection 3.2}, the critical boundary estimates are those for the pure tangential second derivatives as the other second derivatives are already estimated in \cite{JT-oblique-I}. The pure tangential second derivative estimates are proved in Lemma \ref{Lemma 3.1} and Lemma \ref{Lemma 3.2} for nonlinear $\mathcal{G}$ and semilinear $\mathcal{G}$, respectively. Then we obtain the corresponding full second derivative estimates in Theorem \ref{Th3.2} and Theorem \ref{Th3.3}. As an application of Theorem \ref{Th3.2}, we also obtain a full second derivative estimate for the second boundary value problem for the augmented Hessian equation \eqref{1.1} in Corollary \ref{Cor 3.1}. In Section \ref{Section 4}, we first prove some corresponding gradient estimates for admissible solutions which supplement those in \cite {JTX2015} and \cite{JT-oblique-I}. We then establish the existence results for classical admissible solutions for the boundary value problem \eqref{1.1}-\eqref{1.2} in Theorems \ref{Th4.1} and \ref{Th4.2} via the method of continuity.

To avoid too many repetitions, all the definitions and notation in this paper, unless otherwise specified, will follow our part I \cite{JT-oblique-I}.
We also refer the reader to \cite{JT-oblique-I} for a more thorough introduction and  background information.

\section{Barrier constructions}\label{Section 2}
\vskip10pt

In this section, we construct  barrier functions from  admissible functions for general augmented Hessian operators with regular matrices $A$,  which are extensions of our original construction in the Monge-Amp\`ere case \cite{JTY2013}; see also \cite{JT2014,JTY2014}. They play the same role as the function $|x|^2$ in the standard Hessian case and replace the barriers arising from the ``$A$-boundedness''  conditions introduced for Monge-Amp\`ere type equations in \cite{Tru2006,TruWang2009}.

Let $\Omega$ be a bounded domain in Euclidean $n$-space $\mathbb{R}^n$ and $\Gamma$ a convex open set in the linear space of $n\times n$ symmetric matrices $\mathbb{S}^n$, which contains and is closed under addition of the positive cone $\Gamma_n=K^+$. We consider augmented Hessian operators of the form
\begin{equation}\label{3.1}
\mathcal{F}[u] = F(M[u])
\end{equation}
acting on functions $u\in C^2(\Omega)$, whose augmented Hessians $M[u]=D^2u-A(\cdot, u,Du)\in \Gamma$, where $F$ is a non-decreasing function in $C^1(\Gamma)$ and $A$ is a continuous mapping from $\bar\Omega\times\mathbb{R}\times \mathbb{R}^n\rightarrow \mathbb{S}^n$. We assume that $A$ is
 twice differentiable and regular with respect to $p$, with $D^2_pA\in C^0(\bar\Omega\times\mathbb{R}\times \mathbb{R}^n)$. Corresponding to the equivalent form of strict regularity expressed in inequality (1.20) in \cite{JT-oblique-I}, it will be convenient  to express  the regular condition in the form
 \begin{equation} \label{regular}
  A_{ij}^{kl}\xi_i\xi_j\eta_k\eta_l \geq - 2\bar\lambda |\xi||\eta|(\xi\cdot\eta),
\end{equation}
for all $\xi, \eta \in \mathbb{R}^n$, where  $\bar\lambda$ is a non-negative function in
$C^0(\bar\Omega\times \mathbb{R}\times \mathbb{R}^n)$, depending on $D^2_pA$. Then corresponding to inequality (2.4) in \cite{JT-oblique-I}, for any non-negative symmetric matrix $\{F^{ij}\}$ and $\epsilon \in (0,1]$, we have the inequality,
\begin{equation}\label{2.3}
 F^{ij}A_{ij}^{kl}\eta_k\eta_l \geq - \bar \lambda ( \epsilon \mathscr{T}|\eta|^2 + \frac{1}{\epsilon} F^{ij}\eta_i\eta_j),
 \end{equation}
 which will be useful throughout this paper.

The linearized operator of $\mathcal{F}$ is given by
\begin{equation}\label{3.3}
Lv:= F^{ij}(M[u])[D_{ij}v-A_{ij}^k(\cdot,u,Du)D_kv], \quad {\rm for} \  \ v\in C^2(\Omega),
\end{equation}
where $F^{ij}=F_{r_{ij}}=\frac{\partial F}{\partial r_{ij}}$ and $A_{ij}^k=D_{p_k}A_{ij}$. Clearly, the operator $\mathcal{F}$ satisfying F1 is elliptic with respect to $u$ for $M[u]\in \Gamma$ and $L$ is elliptic when $M[u]\in \Gamma$. We henceforth call $u$ admissible in $\Omega$, ($\bar \Omega$), if $M[u]\in \Gamma$ in $\Omega$, ($\bar \Omega$) respectively. We construct barrier functions for $L$ in terms of another given function $\bar u\in C^2(\bar \Omega)$, which is admissible in $\bar \Omega$ with respect to $u$ in the sense that $M_u[\bar u] := D^2 \bar u - A(\cdot, u, D\bar u) \in \Gamma$ in $\bar\Omega$. Clearly if $A$ is independent of $z$, then $M_u[\bar u] = M[\bar u]$ so that $\bar u$ is admissible with respect to $u$ if and only if $\bar u$ is admissible while if $A$ is non-decreasing in $z$, (non-increasing in $z$), then $M_u[\bar u] \ge M[\bar u]$ and $\bar u$ is admissible with respect to $u$ if $\bar u$ is admissible and $\bar u \ge u$, ($\le u$). Our arguments are divided into two cases:

\vspace{2mm}

\noindent (a) $F$ is concave, namely F2 holds. The argument here is the same as in our previous works, except we explicitly avoid the covering argument there. By fixing some $\bar \delta>0$, such that $(M_u[\bar u]-\bar \delta I)(\bar \Omega)\subset \Gamma$ and setting $w_{ij}=D_{ij}u-A_{ij}(\cdot,u,Du)$, $\bar w_{ij}=D_{ij}\bar u-A_{ij}(\cdot,u,D\bar u)$, we have
\begin{equation}\label{3.7}
\begin{array}{rl}
L(\bar u-u) =   \!\! & \!\! \displaystyle \bar \delta \mathscr{T} + F^{ij} \{ [\bar w_{ij} - \bar \delta \delta_{ij} - w_{ij}] +
A_{ij}(\cdot, u, D\bar u) - A_{ij}(\cdot, u,Du) -A^k_{ij} D_k(\bar u-u) \} \\
            \ge \!\! & \!\! \displaystyle \bar \delta \mathscr{T} + F^{ij}  [(\bar w_{ij} - \bar \delta \delta_{ij}) - w_{ij}]+ \frac{1}{2}F^{ij}A_{ij}^{kl}(\cdot,u,\hat p)D_k(\bar u-u)D_l(\bar u-u) \\
            \ge \!\! & \!\! \displaystyle \bar \delta \mathscr{T} + F(M_u[\bar u]-\bar \delta I) -\mathcal{F}[u] + \frac{1}{2}F^{ij}A_{ij}^{kl}(\cdot,u,\hat p)D_k(\bar u-u)D_l(\bar u-u),
\end{array}
\end{equation}
using  Taylor's formula, where $A_{ij}^{kl}=D^2_{p_kp_l}A_{ij}$, $\hat p=\theta Du+(1-\theta)D\bar u$ for some $\theta\in (0,1)$. Note that in the above cases when $M_u[\bar u] \ge M[\bar u]$, we can choose $\bar \delta$ to depend only on $\Omega,\bar u, A$ and $\Gamma$. Otherwise it would also depend on the modulus of continuity of $u$.
By direct calculation, we have
\begin{equation}\label{insert 1 between 37 38}
  \begin{array}{rl}
    Le^{K(\bar u-u)} = \!\! & \!\! \displaystyle  Ke^{K(\bar u-u)}[L(\bar u-u)+ KF^{ij}D_i(\bar u-u)D_j(\bar u-u)] \\
                     = \!\! & \!\! \displaystyle  Ke^{K(\bar u-u)} \{\bar \delta \mathscr{T}  + \frac{1}{2}F^{ij}A_{ij}^{kl}D_k(\bar u-u)D_l(\bar u-u)+ KF^{ij}D_i(\bar u-u)D_j(\bar u-u) \\
                       \!\! & \!\! \displaystyle + F(M_u[\bar u]-\bar \delta I) -\mathcal{F}[u]\},
  \end{array}
\end{equation}
From the regular condition \eqref{2.3}, we obtain
\begin{equation}\label{insert 2 between 37 38}
  \begin{array}{rl}
      \!\! & \!\! \displaystyle \bar \delta \mathscr{T} + \frac{1}{2}F^{ij}A_{ij}^{kl}D_k(\bar u-u)D_l(\bar u-u)+ KF^{ij}D_i(\bar u-u)D_j(\bar u-u)  \\
      \ge \!\! & \!\! \displaystyle  (\bar \delta - \frac{\bar \lambda\epsilon}{2} |D(\bar u-u)|^2)\mathscr{T} + (K-\frac{\bar \lambda}{2\epsilon})F^{ij}D_i(\bar u-u)D_j(\bar u-u) \ge \frac{\bar \delta}{2}\mathscr{T},
  \end{array}
\end{equation}
by successively fixing $\epsilon \le \bar \delta / \sup_{\Omega}(\bar \lambda|D(\bar u-u)|^2)$ and $K\ge \sup_{\Omega}\bar \lambda / (2\epsilon)$.
From \eqref{insert 1 between 37 38} and \eqref{insert 2 between 37 38}, we obtain
\begin{equation}\label{3.8}
\begin{array}{ll}
Le^{K(\bar u-u)} \!\! & \!\! \displaystyle \ge \delta \mathscr{T} + Ke^{K(\bar u-u)} \{ F(M_u[\bar u]-\bar \delta I) -\mathcal{F}[u]\}\\
                 \!\! & \!\! \displaystyle \ge \delta \mathscr{T} - C\mathcal{F}[u] -C',
\end{array}
\end{equation}
as in our previous constructions \cite{JTY2013, JT2014, JTY2014}, where $(M_u[\bar u]-\bar \delta I)(\bar \Omega)\subset \Gamma$ is used in the second inequality,   $K$, $\delta$ and $C^\prime$ are positive constants depending on $n, \Omega, D^2_p A, \bar \delta, |u|_{1;\Omega}$ and $|\bar u|_{1;\Omega}$  and $C^\prime$ is a further constant depending also on $M_u[\bar u]$ and $F$.

\vspace{2mm}

\noindent (b)
$M_u[\bar u]\in K^+$, that is there exists a positive constant $\bar \delta$ such that
\begin{equation}\label{3.4}
M_u[\bar u] \ge \bar \delta I, \quad {\rm in} \ \Omega.
\end{equation}
We then have in place of \eqref{3.7},
\begin{equation}\label{3.5}
\begin{array}{rl}
L(\bar u-u) = \!\! & \!\!\displaystyle  F^{ij}\{\bar w_{ij} - w_{ij} +
               A_{ij}(\cdot,u,D\bar u) - A_{ij}(\cdot,u,Du) - A_{ij}^kD_{k}(\bar u-u)\}\\
         \ge  \!\! & \!\!\displaystyle  F^{ij}[\bar w_{ij}-w_{ij}+\frac{1}{2}A_{ij}^{kl}(\cdot, u, \hat p)D_k(\bar u-u)D_l(\bar u-u)]
\end{array}
\end{equation}
where Taylor's formula is again used, $A_{ij}^{kl}=D^2_{p_kp_l}A_{ij}$, $\hat p=\theta Du+(1-\theta)D\bar u$ for some $\theta\in (0,1)$. Using the regular condition \eqref{2.3} as in \eqref{insert 2 between 37 38}, from \eqref{3.4} and \eqref{3.5}, we obtain
\begin{equation}\label{3.6}
Le^{K(\bar u-u)} \ge \delta \mathscr{T} - C F^{ij}w_{ij},
\end{equation}
where $K$, $\delta$ and $C$ are positive constants depending on $n, \Omega, D^2_p A, \bar \delta, |u|_{1;\Omega}$ and $|\bar u|_{1;\Omega}$.

\begin{remark}\label{Remark2.1}
(i) The constructed function $\eta:=e^{K(\bar u-u)}$ can be regarded as a global barrier function and is important in  second derivative estimates when $A$ is merely regular and not strictly regular; see \cite{JT2014,JTY2014}.

(ii) For the key inequality \eqref{insert 2 between 37 38}, we used the original form \eqref{MTW} of the regular condition on $A$ in our previous papers \cite{JTY2013, JT2014, JTY2014}. While in the above proof, we use the equivalent form \eqref{regular} (or \eqref{2.3})  so that we do not need to assume $D(\bar u-u)=(D_1(\bar u -u), 0, \cdots, 0)$ at any given point in $\Omega$ as before.

(iii) The term $F^{ij}w_{ij}$ in \eqref{3.6} is readily estimated from above in terms of $\mathcal{F}[u]$ under further hypotheses on $F$. In particular if $F$ is homogeneous of degree one, we have again \eqref{3.8} with $C^\prime = 0$.  If $F$ is concave, we recover \eqref{3.8} by taking $\mu = \frac{\delta}{2}$ in inequality (1.9) in \cite{JT-oblique-I}. More generally if $F$ satisfies
\begin{equation}\label{homogeneity}
r\cdot F_r \le F +\mu\mathscr{T} + C_\mu
\end{equation}
in $\Gamma$, for any positive constant $\mu$ and some constant $C_\mu$, depending on $\mu$, then we also obtain \eqref{3.8} in case (b), with $C^\prime$ depending on $C_\mu$ for some $\mu$ depending on $\delta$ and $C$ in \eqref {3.6}. In the special case, $F=\log \det$, $\Gamma=K^+$, we have $F^{ij}w_{ij}=n$, whence the lower bound in \eqref{3.8} is independent of $\mathcal F[u]$. We remark also that in case (b), the matrix $\{F_{ij}\}$ can be replaced by any non-negative matrix function.

(iv) Appropriate admissible functions $\bar u$ exist for matrix functions $A$ arising in optimal transportation or the more general generated prescribed Jacobian equations, see \cite{JT2014}.  For $A=\frac{1}{2}a_{kl}(x,z)p_kp_lI -a_0(x,z)p\otimes p$ with $a_{kl}, a_0\in C^2(\bar\Omega \times \mathbb{R})$ and $\{a_{kl}\}\ge0$, $a_0\ge 0$ in $\bar\Omega \times \mathbb{R}$, quadratic functions $\bar u = c_0+ \frac{1}{2}\epsilon |x-x_0|^2$ will be admissible for arbitrary constants $c_0$ and points $x_0\in \Omega$ and sufficiently small $\epsilon$, see Section 4.2 in \cite{JT-oblique-I}.
\end{remark}

So far, we have only used the operator $\mathcal{F}$ in \eqref{3.1}. Now we assume that $u$ is also a supersolution of equation \eqref{1.1} with a given right hand side $B\in C^0(\bar\Omega\times\mathbb{R}\times \mathbb{R}^n)$, differentiable with respect to $p$ with $D_pB\in C^0(\bar\Omega\times\mathbb{R}\times \mathbb{R}^n)$, and define the corresponding linearized operator
\begin{equation}\label{Linearized op}
\mathcal{L}:= L -D_{p_k}B(\cdot,u,Du)D_k.
\end{equation}
If $F$ satisfies condition F$5^+$, that is $F(r)\rightarrow\infty$ as $r \rightarrow \infty$ uniformly for $F(r) \in \mathcal I$ for any interval $\mathcal I \subset\subset F(\Gamma)$, we obtain from \eqref {3.8},
\begin{equation}\label{2.10}
\mathcal{L}\eta \ge \delta (\mathscr{T} + 1),
\end{equation}
for a further positive constant  $\delta$, depending also on $B$, provided $D^2u$ is sufficiently large. In general without assuming F$5^+$, we can achieve the same inequality in case (a) if F1 holds, that is  $F_r> 0$ in $\Gamma$, $B$ is convex in $p$ and $\bar u$ is also a subsolution of an appropriate equation, in the sense that
\begin{equation}\label{subsolution}
F(M_u[\bar u]) \ge B(\cdot,u,D\bar u), \quad {\rm in} \ \Omega.
\end{equation}
By a standard perturbation argument, using the linearized operator and the mean value theorem, as in \cite{JTY2013},  we can assume the differential inequality \eqref{subsolution} is strict so that  for $\bar \delta$ sufficiently small
$(M[\bar u]-\bar \delta I)(\bar \Omega)\subset \Gamma$ satisfies
\begin{equation}\label{2.12}
F(M_u[\bar u]-\bar \delta I) > B(\cdot, u, D\bar u),  \quad {\rm in} \ \bar\Omega.
\end{equation}
From the first inequality in \eqref{3.8}, we then have for a further positive constant  $\delta$,
\begin{equation}\label{barrier without - C}
\begin{array}{rl}
\mathcal{L}\eta
                \!\!&\!\! \displaystyle \ge \delta  (\mathscr{T} +1) + C\{B(\cdot,u, D\bar u) - B(\cdot, u,Du) - D_{p_k}B(\cdot,u,Du)D_k(\bar u-u)\}\\
                \!\!&\!\! \displaystyle \ge \delta (\mathscr{T} + 1),
\end{array}
\end{equation}
where the convexity of $B$ in $p$ is used to obtain the last inequality.

We summarise the above constructions in the following lemma.

\begin{Lemma}\label{Lemma 2.1}
Let $u\in C^2(\bar \Omega)$ be an admissible function of equation \eqref{1.1} and $\bar u\in C^2(\bar \Omega)$ be  admissible with respect to $u$. Assume $\mathcal{F}$ satisfies F1 and $A$ is regular.

{\rm (i)} If either {\rm (a)} F2 holds or {\rm (b)}  $M_u[\bar u]\in K^+ \subset \Gamma$ and \eqref{homogeneity} holds, then the estimate \eqref{3.8} holds for positive constants $K, \delta, C$ and $C^\prime$ depending on  $n,F, \Omega,  A,  \bar u$ and $|u|_{1;\Omega}$.

{\rm (ii)} If u is a supersolution of equation \eqref{1.1} with $B$ convex in $p$, F2 holds and $\bar u$ satisfies the subsolution condition \eqref{subsolution}, then the estimate \eqref{2.10} holds for positive constants $K$ and $\delta$,
depending on  $n,F, \Omega,  A,B, \bar u$ and $|u|_{1;\Omega}$.
\end{Lemma}

As the applications of Lemma \ref{Lemma 2.1}, the barrier $\eta$ constructed from $\bar u$ and $u$ will be used for the global second derivative estimates in Section \ref{Section 3}, as well as a global gradient estimate in \cite{JT-new}.

\section{Second derivative estimates}\label{Section 3}
\vskip10pt

In this section, we derive global second derivative  estimates  for admissible solutions of the oblique boundary value problem \eqref{1.1}-\eqref{1.2} with regular $A$.

We shall as usual introduce some notational convention  and make some preliminary calculations. We denote partial derivatives of functions on $\Omega$ by subscripts, that is $u_i = D_i u, u_\tau = D_\tau u = \tau_i u_i,
u_{ij} = D_{ij}u, u_{i \tau} = u_{ij}\tau_j, u_{\tau\tau} = u_{ij}\tau_i\tau_j, w_{ij}=u_{ij}-A_{ij}$ etc.
For a constant unit vector $\tau$, differentiating equation \eqref{1.1} in the $\tau$ direction, we have,
\begin{equation}\label{once diff}
\mathcal{L}u_\tau = F^{ij}\tilde D_{x_\tau} A_{ij} + \tilde D_{x_\tau}  B.
\end{equation}
where $\tilde D_{x_\tau} = \tau\cdot\tilde D_x$  and  $\tilde D_x = D_x +DuD_z$, $\mathcal{L}$ is the linearized operator in \eqref{Linearized op}.
Differentiating again in the $\tau$ direction, we then obtain
\begin{equation}
\begin{array}{ll}\label{twice diff}
\mathcal{L}u_{\tau\tau}= \!&\!\!\displaystyle -F^{ij,kl} D_{\tau}w_{ij}D_{\tau}w_{kl} +F^{ij}[\tilde D_{x_\tau x_\tau}A_{ij}+A_{ij}^{kl}u_{k\tau}u_{l\tau}
+2(\tilde D_{x_\tau}A^k_{ij})u_{k\tau}] \\
  \!&\!\!\displaystyle  + (D_{p_kp_l}B)u_{k\tau} u_{l\tau} +\tilde D_{x_\tau x_\tau} B +  2(\tilde D_{x_\tau} D_{p_k} B)u_{k\tau}.
 \end{array}
\end{equation}

\subsection{Global second derivative estimates}\label{Subsection 3.1}

We begin by formulating the global second derivative estimate for equation \eqref{1.1} which follows from our barrier constructions in Section \ref{Section 2}, similarly to the $k$-Hessian case in \cite{JTY2014}.

\begin{Theorem}\label{Th3.1}
Assume that $F$ is orthogonally invariant satisfying F1, F2 and F3, $A\in C^2(\bar \Omega\times \mathbb{R}\times \mathbb{R}^n)$ is regular in $\bar\Omega$, $B > a_0, \in C^2(\bar \Omega\times \mathbb{R}\times \mathbb{R}^n)$ is convex in $p$, $u\in C^4(\Omega)\cap C^2(\bar \Omega)$ is an admissible solution of equation \eqref{1.1} and $\bar u\in C^2(\Omega)\cap C^1(\bar \Omega)$ is admissible with respect to $u$. Assume further, either {\rm (i)}  F5\textsuperscript{+} holds or
{\rm (ii)} $\bar u$ also satisfies the subsolution condition \eqref{subsolution}.  Then we have the following estimate
\begin{equation}\label{global 2nd bound regular A}
\sup_{\Omega}|D^2u| \le C(1+\sup_{\partial\Omega}|D^2u|),
\end{equation}
where the constant $C$ depends on $A, B, F, \Omega, \bar u$ and $|u|_{1;\Omega}$.
\end{Theorem}

\begin{proof}
Let $v$ be an auxiliary
function given by
    \begin{equation}\label{aux function}
        v(x,\xi):=\log (w_{\xi\xi})+ \frac{a}{2}(1+\frac{1}{2}|Du|^2)^2+b\eta,
    \end{equation}
where $w_{\xi\xi}=w_{ij}\xi_i\xi_j=(u_{ij}-A_{ij})\xi_i\xi_j$ with a
vector $\xi\in \mathbb{R}^n$, $\eta=e^{K(\bar u-u)}$ is the
barrier function as in Lemma \ref{Lemma 2.1}, $a$ and $b$ are positive constants to be chosen later, (with $a$ small and $b$ large).

Assume that $v$ takes its maximum at an interior point $x_0 \in
\Omega$ and a unit vector $\xi_0$. Without loss of generality, we can
choose the coordinate system $e_1,\cdots,e_n$ at $x_0$ such that $e_1(x_0)=\xi_0$, $\{w_{ij}(x_0)\}$ is diagonal, and
$w_{11}(x_0)= \max\limits_i {w_{ii}(x_0)}$. Then we will assume $w_{11}(x_0)\ge 1$ as large as want, otherwise we are done. Since the operator $\mathcal{F}$ is orthogonally invariant, $\{F^{ij}\}$ is also diagonal at $x_0$.

Since the function $\vartheta(x):=v(x,e_1)$ attains its maximum at the point $x_0$, we have $D\vartheta(x_0)=0$ and $\mathcal{L}\vartheta(x_0)\le 0$. By direct calculations, we have at the maximum point $x_0$,
    \begin{equation}\label{differentiate once '}
       0 = D_i\vartheta=\frac{D_iw_{11}}{w_{11}}+a(1+\frac{1}{2}|Du|^2)u_ku_{ki}+b\eta_i, \quad {\rm for}\ i=1,\cdots, n,
    \end{equation}
and
    \begin{equation}\label{Lv leq 0}
        \begin{array}{rl}
            0 \ge \mathcal {L}\vartheta
              =   \!\!\! & \displaystyle \frac{1}{w_{11}}\mathcal {L}(w_{11}) -\frac{1}{w^2_{11}}F^{ii}(D_iw_{11})^2 \\
                  \!\!\! & \displaystyle +a(1+\frac{1}{2}|Du|^2)u_i\mathcal {L}u_i+b\mathcal {L}\eta \\
                  \!\!\! & \displaystyle +a(1+\frac{1}{2}|Du|^2)F^{ii}u_{ki}^2+aF^{ii}(u_ku_{ki})^2.
        \end{array}
    \end{equation}
We shall successively estimate the terms on the right hand side of (\ref{Lv leq 0}). Note that all the calculations will be made at the point $x_0$. By replacing $\tau$ with $e_1$ in \eqref{twice diff}, we have
    \begin{equation}\label{L u11}
       \begin{array}{rl}
          \mathcal{L}u_{11} \ge \!\!&\!\! - F^{ij,kl}D_1w_{ij}D_1w_{kl} - C(1+w_{11})(1+\mathscr{T}) \\
                                \!\!&\!\! + F^{ii}A_{ii}^{kl}u_{1k}u_{1l} + (D_{p_kp_l}B)u_{1k}u_{1l}.
       \end{array}
    \end{equation}
Here the constant $C$ depends on $A, B, \Omega$ and $|u|_{1;\Omega}$. Unless otherwise specified, we shall use $C$ to denote a positive constant
with such dependence in this proof. Using the regularity condition \eqref{MTW}, we can estimate
    \begin{equation}\label{MTW term}
        \begin{array}{rl}
            F^{ii}A_{ii}^{kl}u_{1k}u_{1l} \!\!\! & =
            F^{ii}A_{ii}^{kl}(w_{1k}+A_{1k})(w_{1l}+A_{1l})\\
            \!\!\! & \geq F^{11}A_{11}^{11}w^2_{11}-C\mathscr{T}(1+w_{11})\\
            \!\!\! & \geq -CF^{11}w^2_{11}-C\mathscr{T}(1+w_{11}).
        \end{array}
    \end{equation}
By a direct calculation, we have
    \begin{equation}\label{L A 11}
        \begin{array}{rl}
            -\mathcal {L}A_{11}  \ge \!\!&\!\!
            -F^{ii}A_{11}^{ii}w_{ii}^2 - F^{ii}A_{11}^k u_{kii} - C(1+w_{11})(1+\mathscr{T})\\
            \leq \!\!&\!\!
            -CF^{ii}w_{ii}^2 - C(1+w_{11})(1+\mathscr{T}),
        \end{array}
    \end{equation}
where the third derivative term in the first line is treated by using the once differentiated equation \eqref{once diff}. Combining \eqref{L u11}, \eqref{MTW term}, \eqref{L A 11}, and using the convexity of $B$ in $p$, we obtain
    \begin{equation}\label{first and second terms}
        \begin{array}{rl}
               \!\!&\!\!\displaystyle \frac{1}{w_{11}}\mathcal {L}(w_{11}) -\frac{1}{w^2_{11}}F^{ii}(D_iw_{11})^2 \\
           \ge \!\!&\!\!\displaystyle -\frac{1}{w_{11}}F^{ij,kl}D_1w_{ij}D_1w_{kl} -\frac{1}{w^2_{11}}F^{ii}(D_iw_{11})^2 \\
               \!\!&\!\!\displaystyle -\frac{C}{w_{11}}F^{ii}w_{ii}^2 - C(1+\mathscr{T}).
        \end{array}
    \end{equation}
We need to estimate the first two key terms on the right hand side of \eqref{first and second terms} which involve third derivatives. For this, by setting
    \begin{equation*}
    I=\{i\ \!|\! \ w_{ii}\leq -\frac{1}{3} w_{11}\},\ \ J=\{j\ \!|\! \ w_{jj}>-\frac{1}{3} w_{11}, j>1\},
    \end{equation*}
we have $1 \notin I$, $1 \notin J$, $I\cap J=\emptyset$, $\{1\}\cup
I \cup J=\{1,\cdots,n\}$, and $J^c={1}\cup I$, where $J^c$ is the
complementary set of $J$. By Andrews formula in Lemma 2.2 in \cite{JTY2014}, (see also \cite{B.Andrews, G}), and F2, we have
    \begin{equation}\label{Key third derivatives}
        \begin{array}{ll}
              \!\!&\!\! \displaystyle -\frac{1}{w_{11}}F^{ij,kl}D_1w_{ij}D_1w_{kl} -\frac{1}{w^2_{11}}F^{ii}(D_iw_{11})^2 \\
         \ge \!\!&\!\! \displaystyle \frac{2}{w_{11}}\sum_{i\neq j}\frac{F^{ii}-F^{jj}}{w_{jj}-w_{ii}}(D_1w_{i1})^2 -\frac{1}{w^2_{11}}F^{ii}(D_iw_{11})^2\\
         \ge \!\!&\!\! \displaystyle \frac{1}{w^2_{11}}\sum_{i\in J}F^{ii}(D_iw_{11})^2-\frac{C}{w^2_{11}}\sum_{i\in J}F^{ii}-\frac{F^{11}}{w^2_{11}}\sum_{i\in J}(D_iw_{11})^2-\frac{1}{w^2_{11}}F^{ii}(D_iw_{11})^2 \\
         \ge \!\!&\!\! \displaystyle -\sum_{i\in I}F^{ii}(\frac{D_iw_{11}}{w_{11}})^2-F^{11}\sum_{i\notin I}(\frac{D_iw_{11}}{w_{11}})^2 -C\mathscr{T},
        \end{array}
    \end{equation}
where the second inequality is obtained by fixing $j=1$, commuting $D_1w_{i1}$ and $D_iw_{11}$, and using Cauchy's inequality, (see (3.16) in \cite{JTY2014}).
Then using \eqref{differentiate once '} in \eqref{Key third derivatives}, we have
    \begin{equation}\label{Key third derivatives'}
        \begin{array}{ll}
             \!\!&\!\! \displaystyle -\frac{1}{w_{11}}F^{ij,kl}D_1w_{ij}D_1w_{kl} -\frac{1}{w^2_{11}}F^{ii}(D_iw_{11})^2 \\
         \ge \!\!&\!\! \displaystyle - C \sum_{i\in I}F^{ii}(a^2w_{ii}^2 +a^2+ b^2) - CF^{11}\sum_{i\notin I}(a^2w_{ii}^2+a^2+b^2)-C\mathscr{T} \\
         \ge \!\!&\!\! \displaystyle -a^2CF^{ii}w_{ii}^2 - b^2C\sum_{i\in I}F^{ii} - a^2C F^{11}w_{11}^2 - C(a^2+b^2)F^{11} -C(1+a^2)\mathscr{T},
        \end{array}
    \end{equation}
where the properties of $I$ and $J$ are used in the last inequality. Now returning to \eqref{Lv leq 0} and using \eqref{once diff}, we have
    \begin{equation}\label{second line estimate 1}
        a(1+\frac{1}{2}|Du|^2)u_i\mathcal {L}u \ge - aC (1+\mathscr{T}).
    \end{equation}
By the barrier construction in Lemma \ref{Lemma 2.1}, if either {\rm (i)} F5\textsuperscript{+} holds or
(ii) $\bar u$ also satisfies the subsolution condition \eqref{subsolution}, we have
    \begin{equation}\label{second line estimate 2}
        b\mathcal{L}\eta \ge b\delta (\mathscr{T} + 1),
    \end{equation}
where the constant $\delta$ depends on  $n, F, \Omega, A, B, \bar u$ and $|u|_{1;\Omega}$.
The last two terms in \eqref{Lv leq 0} are readily estimated as follows.
    \begin{equation}\label{third line estimate}
       a(1+\frac{1}{2}|Du|^2)F^{ii}u_{ki}^2+aF^{ii}(u_ku_{ki})^2\ge \frac{a}{2}F^{ii}w_{ii}^2 - aC\mathscr{T}.
    \end{equation}
Consequently, by \eqref{Lv leq 0}, \eqref{first and second terms}, \eqref{Key third derivatives'}, \eqref{second line estimate 1}, \eqref{second line estimate 2} and \eqref{third line estimate}, we have
    \begin{equation}\label{Lv leq 0 ''}
            0\ge \mathcal {L}\vartheta \ge \frac{a}{8}F^{ii}w_{ii}^2 + \frac{b\delta}{2}(1+\mathscr{T}) -b^2C \sum_{i\in I}F^{ii} - a^2CF^{11}w_{11}^2 - C(a^2+b^2)F^{11}.
    \end{equation}
Note that in order to obtain \eqref{Lv leq 0}, we have fixed $a=1/(16C)$, $b=2C(1+a+a^2)/\delta$, and assumed $w_{11}\ge 72b^2C/a$. From the choice of $a$, $b$, we have
    \begin{equation}\label{final control}
        \begin{array}{rl}
        \displaystyle
            \frac{a}{8}F^{ii}w_{ii}^2 \ge \!\!&\!\!\displaystyle \frac{a}{8}F^{11}w_{11}^2 + \sum_{i\in I} F^{ii}w_{ii}^2 \\
                                      \ge \!\!&\!\!\displaystyle \frac{a}{8}F^{11}w_{11}^2 + \frac{a}{72}w_{11}\sum_{i\in I}F^{ii} \\
                                      \ge \!\!&\!\!\displaystyle \frac{a}{16}F^{11}w_{11}^2+a^2 C F^{11}w_{11}^2 + b^2C\sum_{i\in I}F^{ii}.
        \end{array}
    \end{equation}
We then obtain from \eqref{Lv leq 0 ''} and \eqref{final control},
    \begin{equation*}
       w_{11} \le \sqrt{\frac{16C(a^2+b^2)}{a}},
    \end{equation*}
which implies the conclusion \eqref{global 2nd bound regular A} and completes the proof of Theorem \ref{Th3.1}.
\end{proof}

\begin{remark}\label{Remark3.1}
If the matrix $A$ has the form
\begin{equation}\label{special A}
A(x,z,p) = A_0(x,z) + A_1(p)
\end{equation}
satisfying the strengthened regular condition
\begin{equation}\label{strengthened regularity}
A_{ij}^{kl}\xi_i\xi_j\eta_k\eta_l \geq - \bar\lambda (\xi\cdot\eta)^2,
\end{equation}
for all $\xi, \eta \in \mathbb{R}^n$, where  $\bar\lambda$ is a non-negative function in
$C^0(\Omega\times \mathbb{R}\times \mathbb{R}^n)$, and $B$ has the form
\begin{equation}\label{special B}
B(x,z,p) = B_0(x,z) + B_1(p),
\end{equation}
then the orthogonal invariance of $\mathcal{F}$ is not needed in Theorem \ref{Th3.1}. To show this we consider the auxiliary function
\begin{equation}\label{aux in Remark 3.1}
v(x,\xi):=u_{\xi\xi} + \frac{a}{2}|u_\xi|^2+b\eta
\end{equation}
in place of \eqref{aux function}. By calculation, we have
    \begin{equation}\label{inequal in Rm2.1}
        \begin{array}{ll}
           \mathcal{L}v \!\!&\!\!\displaystyle \ge - C(1+\mathscr{T}) + F^{ij}A_{ij}^{kl}u_{\xi k}u_{\xi l} + aF^{ij}u_{\xi i}u_{\xi j} + \frac{b\delta}{2}(1+\mathscr{T}) \\
                        \!\!&\!\!\displaystyle \ge (a-\bar \lambda)F^{ij}u_{\xi i}u_{\xi j} + \frac{b\delta}{4}(1+\mathscr{T})>0,
        \end{array}
    \end{equation}
by choosing $b\ge \frac{4C}{\delta}$ and $a >\sup \bar\lambda$, which implies
    \begin{equation}\label{improved 2nd global estimate}
       \sup\limits_{\Omega}|D^2u|\le \sup_{\partial\Omega}|D^2u| + C,
    \end{equation}
where $C$ depends on $A, B, F, \Omega, \bar u$ and $|u|_{1;\Omega}$.
Therefore, we obtain the global second derivative estimate without the orthogonal invariance of $\mathcal{F}$ for $A$ satisfying \eqref{special A} and \eqref{strengthened regularity} and $B$ satisfying \eqref{special B}.

We also remark that if $A=A_0(x,z)$ and $B=B_0(x,z)$ are independent of $p$, instead of $\eta=e^{K(\bar u-u)}$, we can
take $\eta(x)=|x|^2$ in \eqref{aux in Remark 3.1} since then $\mathcal{L}\eta =2\mathscr{T}$.
\end{remark}

\begin{remark}\label{Remark3.2}
For the $k$-Hessian operators $\mathcal F_k$ we do not need to assume the convexity of $B$ with respect to $p$ in the special cases $k=1,2$ or $n$ in Theorem \ref{Th3.1}. The Monge-Amp\`ere case, $k=n$, is already covered in \cite{JTY2013, JTX2015}, with simpler proof, while in the case, $k=2$, the inequality, $|r| \le \mathscr T(r)$ in $\Gamma_2$ shows that convexity in $p$ is not needed to infer \eqref{first and second terms} from \eqref{L u11}, \eqref{MTW term} and \eqref{L A 11}.
\end{remark}

\begin{remark}\label{Remark-add at the end}
When $\mathcal{F}$ is orthogonally invariant, if we diagonalise $\{w_{ij}\}$ at a point $x_0$, then $\{F^{ij}\}$ is also diagonal at $x_0$. From the regularity of $A$, we have at $x_0$,
\begin{equation}\label{inequal in Rm at the end}
F^{ij}A_{ij}^{kl}w_{1k}w_{1l}=F^{ii}A_{ii}^{kl}w_{1k}w_{1l} \ge F^{11}A_{11}^{11}w_{11}^2 \ge -CF^{11}w_{11}^2, 
\end{equation}
which is used in the key estimate \eqref{MTW term}. Comparing with \eqref{2.3}, the estimate \eqref{inequal in Rm at the end} is better for the proof of Theorem \ref{Th3.1}, since the first term on the right hand side of \eqref{2.3} does not appear when $F$ is orthogonally invariant.
\end{remark}

\subsection{Boundary estimates for second derivatives}\label{Subsection 3.2}

Theorem \ref{Th3.1} reduces global bounds for second derivatives for the boundary value problem to boundary estimates, which we now take up.
For the bounded domain $\Omega\subset \mathbb{R}^n$ with $\partial\Omega\in C^2$, we define the $A$-curvature matrix $K_A[\partial\Omega]$ and the uniform $(\Gamma,A, G)$-convexity as in \cite{JT-oblique-I}.

To compensate for the matrix  $A$ not being strictly regular as in \cite{JT-oblique-I}, we will impose either a stronger concavity condition on the function $G$ in \eqref{1.2} with respect to the gradient variables or stronger monotonicity conditions on $G$ or $A$ with respect to the solution variable. For this purpose we call the function $G\in C^2(\partial\Omega \times \mathbb{R}\times \mathbb{R}^n)$ locally uniformly concave, (concave), in $p$
if $D^2_p G < 0, ( \le 0),$ in $\partial\Omega\times \mathbb{R}\times \mathbb{R}^n$.  If $u\in C^1(\bar \Omega)$ and
\begin{equation}\label{uniform concave of G}
G_{p_kp_l}(\cdot, u, Du)\xi_k\xi_l \le - \gamma_0 |\xi|^2, \quad {\rm on}  \ \partial\Omega,
\end{equation}
for all $\xi\in \mathbb{R}^n$ and some constant $\gamma_0>0,  (= 0)$, then we say the function $G$ is  uniformly concave, (concave) in $p$, with respect to $u$.
The boundary condition \eqref{1.2} is said to be quasilinear if $G$ in \eqref{1.2} has the form
\begin{equation}\label{quasilinear}
G(x,z,p) = \beta(x,z)\cdot p - \varphi(x,z),
\end{equation}
where $\beta$ and $\varphi$ on $\partial\Omega\times \mathbb{R}$. When $\beta$ in \eqref{quasilinear} is independent of $z$, we say the boundary condition \eqref{1.2} is semilinear.

First we consider the estimation of the non-tangential second derivatives.
By tangential differentiation of the boundary condition \eqref{1.2}, the equality (2.21) in \cite{JT-oblique-I} holds and hence we get the mixed tangential oblique derivative estimate
\begin{equation}\label{mixed nonlinear G}
|u_{\tau\beta}|\le C, \quad {\rm on} \ \partial\Omega,
\end{equation}
for any unit tangential vector field $\tau$, where $\beta=D_pG(\cdot, u, Du)$ and the constant $C$ depends on $G$, $\Omega$ and $|u|_{1;\Omega}$.
When conditions F1-F5 and F6 hold, $\mathcal{G}$ is oblique, satisfying \eqref{obliqueness}, and either F5\textsuperscript{+} holds or $B$ is independent of $p$, then $F, A, B$ and $G$ satisfy the hypotheses  of Lemma 2.2 in \cite{JT-oblique-I} and we obtain  upper bounds for the pure second order oblique derivatives from that lemma, namely
\begin{equation}\label{upper bound for double oblique}
\sup\limits_{\partial\Omega} u_{\beta\beta} \le \epsilon M_2 + C_\epsilon,
\end{equation}
for any $\epsilon>0$, where $M_2=\sup\limits_{\Omega}|D^2u|$, and $C_\epsilon$ is a constant depending on $\epsilon, F, A, B, G, \Omega, \beta_0$ and $|u|_{1;\Omega}$. Moreover in the case when $B_p = 0$, we can take $\epsilon = 0$, in \eqref{upper bound for double oblique}.

We next consider the pure tangential derivative estimates on the boundary when $G$ in \eqref{1.2} is uniformly concave in $p$, with respect to $u$. In this paper however we will only consider the cases when $A$ is affine in $p$ on $\partial\Omega$, that is $D_{pp}A=0$ on $\partial\Omega$, or $\mathcal G \ge 0$, near $\partial \Omega$, although a more general situation, depending on gradient bounds, is embraced by condition (1.19) in Remark 1.2 in \cite{JT-oblique-I}. If $D_{pp}A=0$ on $\partial\Omega\times\mathbb{R}\times \mathbb{R}^n$, it follows that the $A$-curvature matrix $K_A[\partial\Omega]$ is independent of $p$ so that the notion of $(\Gamma, A, G)$-convexity is independent of $G$, and can be written simply as $(\Gamma, A)$-convexity. Also note that we can assume without any loss of generality that $G$ has been extended smoothly to $\bar\Omega\times \mathbb{R}\times \mathbb{R}^n$, in which case we may more generally call $G$ uniformly concave in $p$, on a subset $\mathcal N\in \bar\Omega$, with respect to $u$, if \eqref{uniform concave of G} holds on $\mathcal N$. If $G$ is locally uniformly concave in $p$, then $G$ is uniformly concave in $p$, with respect to $u$ in some neighbourhood of $\partial\Omega$.
\begin{Lemma}\label{Lemma 3.1}
Let $u\in C^2(\bar \Omega)\cap C^4(\Omega)$ be an admissible solution of the boundary value problem \eqref{1.1}-\eqref{1.2} in a $C^{2,1}$ domain $\Omega \subset \mathbb{R}^n$, which is uniformly $(\Gamma, A, G)$-convex with respect to $u$. Assume that $F$ satisfies F1-F5, $A\in C^2(\bar \Omega\times \mathbb{R}\times\mathbb{R}^n)$ is regular, $B>a_0, \in C^2(\bar\Omega\times \mathbb{R}\times \mathbb{R}^n)$, $G\in C^2(\bar\Omega\times \mathbb{R}\times \mathbb{R}^n)$ is oblique and uniformly concave in $p$, with respect to $u$ on $\partial\Omega$. Assume also either F5\textsuperscript{+} holds or $B$ is independent of $p$, either $D_{pp}A = 0$ in $\partial\Omega\times\mathbb{R}\times \mathbb{R}^n$ or $\mathcal G [u] \ge 0$ in some neighbourhood $\mathcal N$ of $\partial \Omega$, and either $\mathcal F$ is orthogonally invariant or F6 holds or $G$ is uniformly concave in $p$, with respect to $u$ in $\mathcal N$. Then for any tangential vector field $\tau$, $|\tau|\le 1$, we have the estimate
\begin{equation}\label{pure tangential est}
M_2^+(\tau) \le \epsilon M_2 + C_\epsilon
\end{equation}
for any $\epsilon >0$, where $M_2^+(\tau)=\sup\limits_{\partial\Omega}u_{\tau\tau}$ and $C_\epsilon$ is a constant depending on $\epsilon, F, A, B, G, \beta_0, \Omega$ and $|u|_{1;\Omega}$.
\end{Lemma}
\begin{proof}
In order to obtain the upper bound \eqref{pure tangential est} for the pure tangential second derivatives when $A$ is regular, we need to use the full strength of the uniform concavity of $G$ in $p$ together with the barrier constructed in Lemma 2.2 in \cite{JT-oblique-I} from the uniform $(\Gamma, A, G)$-convexity of $\Omega$. The latter condition combined with either of the conditions, $\mathcal G [u] \ge 0$ in $\mathcal N$ or $D_{pp}A = 0$ in $\partial\Omega\times\mathbb{R}\times \mathbb{R}^n$ implies for a smaller neighbourhood $\mathcal N$ there exists a barrier
$\phi \in C^2( \mathcal N)$ satisfying
\begin{equation}\label{barrier}
L\phi \le - \sigma \mathscr{T},
\end{equation}
in $\mathcal N$, $\phi=0$ on $\partial\Omega$, $\phi > 0$ in $\mathcal N\cap\Omega$, where $\sigma$ is a small positive constant. Moreover we may fix $\mathcal N = \Omega_\rho =\{x\in \Omega|\ d(x)<\rho\}$ for a further small positive constant $\rho$ and $\phi = d-d^2/2\rho$, where $d$ given by $d(x)={\rm dist}(x,\partial\Omega)$ denotes the distance function in $\Omega$. From \eqref{barrier} we then have for the full linearized operator $\mathcal L$,
\begin{equation}\label{full barrier}
\mathcal L\phi \le - \frac{1}{2}\sigma \mathscr{T},
\end{equation}
in $\mathcal{N}$, provided $|D^2u| \ge C_\sigma$, for a further positive constant $C_\sigma$, depending on $\sigma$ and $B$, if F5\textsuperscript{+} holds while $\mathcal L = L$ satisfies \eqref{barrier} if $B_p = 0$. Accordingly the hypotheses of Lemma 3.1 are reduced to the last three alternatives which we now consider.
Following the proof of Lemma 2.3 in \cite{JT-oblique-I}, we suppose that the function
\begin{equation}\label{v tau}
v_\tau = u_{\tau\tau} + \frac{c_1}{2} |u_\tau|^2
\end{equation}
takes a maximum over $\partial\Omega$ and tangential vectors $\tau$, such that $|\tau|\le 1$, at a point $x_0 \in \partial\Omega$ and vector $\tau=\tau_0$, where $c_1$ is now a fixed constant to be determined. Without loss of generality, we may assume $x_0=0$ and $\tau_0=e_1$. For $\beta=D_pG(\cdot, u, Du)$, setting
\begin{equation}\label{b}
b=\frac{\nu_1}{\beta\cdot \nu}, \quad\quad \tau=e_1 - b\beta,
\end{equation}
we then have, at any point in $\partial\Omega$,
\begin{equation}\label{decomposition}
v_1 = v_\tau + b(2u_{\beta\tau}+c_1u_\beta u_\tau)+ b^2 (u_{\beta\beta} + \frac{c_1}{2}u_\beta^2),
\end{equation}
with $v_1(0)=v_\tau (0)$, $b(0)=0$ and $\tau(0)=e_1$. For the cases when F6 holds or $G$ is uniformly concave in $p$ with respect to $u$ near $\partial\Omega$, we can take $c_1=0$, similarly to the F6 case in Lemma 2.3 in \cite{JT-oblique-I}. By \eqref{obliqueness}, \eqref{mixed nonlinear G}, we then have from \eqref{decomposition}, following (2.53) to (2.54) in \cite{JT-oblique-I},
\begin{equation}\label{boundary inequality}
\begin{array}{ll}
v_1 = u_{11} \!\!&\!\! \le u_{\tau\tau} + g(0)\nu_1 + C(1+M_2) |x|^2 \\
\!\!&\!\! \le |\tau|^2 u_{11}(0) + g(0)\nu_1 + C(1+M_2)|x|^2 \\
\!\!&\!\! \le (1-2b\beta_1) u_{11}(0) + g(0)\nu_1 + C(1+M_2) |x|^2, \quad {\rm on} \ \partial\Omega,
\end{array}
\end{equation}
where $g = \frac{2u_{\beta\tau}}{\beta\cdot\nu}$, $|g(0)|\le C$, and $C$ is a constant depending on $\Omega, G, \beta_0$ and $|u|_{1;\Omega}$. For $\epsilon > 0$, we now define truncation functions $\zeta_\epsilon, \tilde\zeta_\epsilon \in
C^2(\mathbb{R})$ such that $\zeta_\epsilon(t) = \epsilon t /|t|$ for $|t| \ge 3\epsilon/2$, $\zeta_\epsilon(t) = t$ for $|t|\le \epsilon/2$, $\tilde\zeta_\epsilon(t) = t$ for $t \ge 3\epsilon/2$, $\tilde\zeta_\epsilon(t) = \epsilon$ for $t \le \epsilon/2$,
$0\le \zeta^\prime_\epsilon$, $\tilde\zeta^\prime_\epsilon\le 1$ and $|\zeta^{\prime\prime}_\epsilon|\le2/\epsilon$, $0\le \tilde\zeta^{\prime\prime}_\epsilon\le2/\epsilon$. With $\nu$ also extended smoothly to all of $\bar \Omega$ so that $\nu$ is constant along normals to $\partial\Omega$ in $\mathcal N$, we can then define for, $0<\epsilon_1 < 1$,
\begin{equation}\label{chi}
\chi(\cdot,u,Du) = 1-2\zeta_{\epsilon_1}(\nu_1) \frac{\beta_1}{\tilde\zeta_{\beta_0/2}(\beta\cdot\nu)}(\cdot,u,Du).
\end{equation}
Taking $\epsilon_1$ small enough depending on $\Omega, G, \beta_0$ and $|u|_{1;\Omega}$, we have $\chi \ge 1/2$ in $\mathcal{N}$.
Clearly $\chi$ is well defined in $C^2 (\bar\Omega)$ and moreover we then have from \eqref{boundary inequality}
\begin{equation}\label{3.31}
v_1 = u_{11} \le \chi(x,u,Du) u_{11}(0) +g(0)\nu_1 + C_1(1 + M_2 +\epsilon_1^{-2}u_{11}(0)) |x|^2 := f , \quad {\rm on} \ \partial\Omega,
\end{equation}
for a further constant $C_1$ depending on $\Omega, G, \beta_0$ and $|u|_{1;\Omega}$, where $f(0)=u_{11}(0)$.
We shall define an auxiliary function
\begin{equation}\label{aux in Lemma 3.1}
v:= v_1 - f  - \alpha M_2 G- \kappa M^2_2 \phi,
\end{equation}
for positive constants $\alpha$ and $\kappa$ to be determined, where $v_1$, $f, G$ and $\phi$ are functions in \eqref{v tau}, \eqref{3.31}, \eqref{1.2} and \eqref{barrier}.
Assuming that the function $v$ attains its maximum over $\bar \Omega_\rho$ at some point $x^0\in \Omega_\rho$, we have $\mathcal{L}v(x^0)\le 0$. We can assume that $|D^2u(x^0)|\ge C_\sigma$ as large as we want, otherwise from $v$ constructed in \eqref{aux in Lemma 3.1}, $G=\phi=0$ on $\partial\Omega$, $\mathcal{G}[u]\ge 0$, $\phi>0$ and $\chi\ge 1/2$ in $\Omega_\rho$, we can obtain the upper bound for $u_{11}(0)$ and derive the desired estimate \eqref{pure tangential est}. Taking $\tau=e_1$ in \eqref{twice diff} and using F2, we have
\begin{equation}\label{3.36}
\begin{array}{ll}
\mathcal{L}u_{11} \ge \!&\!\!\displaystyle F^{ij}[\tilde D_{x_1x_1}A_{ij}+A_{ij}^{kl}u_{1k}u_{1l}
+2(\tilde D_{x_1}A^k_{ij})u_{1k}] \\
  \!&\!\!\displaystyle  + (D_{p_kp_l}B)u_{1k} u_{1l} +\tilde D_{x_1x_1} B +  2(\tilde D_{x_1} D_{p_k} B)u_{1k}.
 \end{array}
\end{equation}
Using the regularity condition \eqref{2.3} of $A$, we have
\begin{equation}\label{3.37}
\begin{array}{ll}
   F^{ij}A_{ij}^{kl}u_{1k}u_{1l} \!\!&\!\!\displaystyle \ge -\bar \lambda(\epsilon \mathscr{T} |Du_1|^2 + \frac{1}{\epsilon}F^{ij}u_{1i}u_{1j})\\
                                 \!\!&\!\!\displaystyle \ge -\bar \lambda(\epsilon \mathscr{T} M_2^2 + \frac{1}{\epsilon}\mathcal{E}_2^\prime),
\end{array}
\end{equation}
for any $\epsilon >0$, where $\bar \lambda$ is a non-negative function, and $\mathcal{E}_2^\prime = F^{ij}u_{ki}u_{kj}$. Then we have
\begin{equation}\label{3.38}
\mathcal{L}u_{11}(x^0) \ge -(\epsilon M_2^2 + C_\epsilon)\mathscr{T} - \frac{C}{\epsilon}\mathcal{E}_2^\prime,
\end{equation}
for either $|D^2u(x^0)|\ge C_\sigma$ if F5\textsuperscript{+} holds, or $B_p=0$. Here in \eqref{3.38}, $\epsilon$ is a further constant which we may take it to be $2\epsilon \max\bar\lambda$ with $\epsilon$ in \eqref{3.37}, and $C$ is a constant depending on $\bar\lambda$. From (2.27) in \cite{JT-oblique-I}, we have
\begin{equation}\label{L chi}
\begin{array}{ll}
   \mathcal{L}\chi \!\!&\!\! \ge F^{ij}(D_{p_kp_l}\chi)u_{ik}u_{jl} - C \mathscr{T} + F^{ij}\tilde \beta_{ik}^\chi u_{jk}\\
         \!\!&\!\! \ge - \epsilon_1 C\mathcal{E}_2^\prime - (\epsilon M_2 + C_\epsilon)\mathscr{T},
\end{array}
\end{equation}
where $\tilde \beta_{ik}^\chi:=2\tilde D_{x_i}D_{p_k}\chi + (D_z\chi)\delta_{ik}$, $\epsilon_1$ is the constant in \eqref{chi}, and the Cauchy's inequality is used in the second inequality. Consequently, we have
\begin{equation}\label{L f}
  \mathcal{L}f \ge - \epsilon_1 CM_2\mathcal{E}_2^\prime - (\epsilon M_2^2 + C_\epsilon)\mathscr{T}.
\end{equation}

In the F6 case, it is enough to take $\alpha=0$.
For $\mathcal{E}_2^\prime = F^{ij}u_{ik}u_{jk}$ and $\mathcal{E}_2=F^{ij}w_{ik}w_{jk}$, we have from the Cauchy's inequality,
\begin{equation}\label{E2' and E2}
  \begin{array}{rl}
    \mathcal{E}_2^\prime \!\!&\!\!\displaystyle  = F^{ij}(w_{ik}+A_{ik})(w_{jk}+A_{jk})\le 2 \mathcal{E}_2 + C \mathscr{T}.
  \end{array}
\end{equation}
Taking \eqref{E2' and E2} into account and using F6, from \eqref{3.38} and \eqref{L f}, we have
\begin{equation}\label{L u11 Lf using F6}
   \mathcal{L}(v_1-f) \ge -4(\epsilon M_2^2 + C_\epsilon)\mathscr{T}, \quad {\rm at} \ x^0.
\end{equation}
By choosing $\kappa=8\epsilon/\sigma$, from \eqref{full barrier} and \eqref{L u11 Lf using F6}, we then have
\begin{equation}\label{Lv >0}
\mathcal{L}v(x^0)>0.
\end{equation}

In the case when $G$ is uniformly concave in $p$ with respect to $u$ near $\partial\Omega$, we have
\begin{equation}\label{L G}
\begin{array}{ll}
   \mathcal{L}G \!\!&\!\! \ge F^{ij}(D_{p_kp_l}G)u_{ik}u_{jl} - C \mathscr{T} + F^{ij}\tilde \beta_{ik}^G u_{jk}\\
      \!\!&\!\! \ge \gamma_0 \mathcal{E}_2^\prime - (\epsilon M_2 + C_\epsilon)\mathscr{T},
\end{array}
\end{equation}
where $\tilde \beta_{ik}^G:=2\tilde D_{x_i}D_{p_k}G + (D_zG)\delta_{ik}$, $\gamma_0$ is the constant in \eqref{uniform concave of G}, and the Cauchy's inequality is used in the second inequality. Without loss of generality, we can assume $M_2 \ge \epsilon_1/\epsilon$, otherwise we have already obtained the second derivative bound. By choosing $\alpha = 2\epsilon_1C/\gamma_0$, and $\kappa =4\epsilon(2+\alpha)/\sigma$, from \eqref{full barrier}, \eqref{3.38}, \eqref{L f} and \eqref{L G}, we also obtain \eqref{Lv >0}.

The inequality \eqref{Lv >0} in both the above cases implies that $v$ attains its maximum at the boundary of $\Omega_\rho$. From \eqref{3.31} and $G=\phi=0$ on $\partial\Omega$, we have $v\le 0$ on $\partial\Omega$ and $v$ attains its maximum $0$ over $\partial\Omega$ at the point $0\in \partial\Omega$. By the choices of $\alpha$, $\kappa$ and assuming $M_2$ sufficiently large, we have $v<0$ on the inner boundary $\Omega\cap \partial\Omega_\rho$, otherwise we can have a upper bound for $M_2$ and finish the proof. Therefore, $v$ must attain its maximum in $\bar \Omega_\rho$ at the point $0\in \partial\Omega$. Thus, we have
\begin{equation}
   D_\beta v(0) \le 0,
\end{equation}
which gives
\begin{equation}\label{3.46}
   u_{11\beta}(0) \le C u_{11}(0) + \alpha M_2 u_{\beta\beta}(0)+ \kappa  M_2^2 (\beta\cdot \nu)(0),
\end{equation}
where $C$ is a constant depending on $\Omega, G, \beta_0$ and $|u|_{1;\Omega}$.
When F6 holds, we have $\alpha=0$, while when $G$ is uniformly concave in $p$ respect to $u$ in $\mathcal{N}$, we can take $\epsilon_1=\epsilon$, so that \eqref{3.46} becomes
\begin{equation}\label{3.47}
   u_{11\beta}(0) \le \epsilon C M_2^2 + C_\epsilon.
\end{equation}
On the other hand, by differentiating the boundary condition \eqref{1.2} twice in a tangential direction $\tau$, we have
\begin{equation}\label{twice diff bdy}
u_{\tau\tau \beta} = - (D_{p_kp_l}G)u_{k\tau}u_{l\tau} - \tilde D_{x_\tau x_\tau}G - 2 (D_{p_k}\tilde D_{x_\tau}G)u_{k\tau}, \quad {\rm on}\ \partial\Omega,
\end{equation}
where $\beta=G_{p}(\cdot,u,Du)$, $\tilde D_{x_\tau} = \tau\cdot\tilde D_x$ and $\tilde D_x = D_x +DuD_z$. At the point $0$, taking $\tau=e_1$ in \eqref{twice diff bdy}, we have
\begin{equation}\label{3.48}
\begin{array}{ll}
u_{11\beta}(0) \!\!&\!\!\displaystyle \ge \gamma_0 u_{11}^2(0) - C(1+M_2),
\end{array}
\end{equation}
by using the uniform concavity of $G$ in $p$ respect to $u$, where the constant $C$ depends on $G, \Omega$ and $|u|_{1;\Omega}$.
Combining \eqref{3.47} and \eqref{3.48}, we then obtain
\begin{equation}\label{contra 2}
u_{11}(0)\le \sqrt{{2\epsilon C}/{\gamma_0}} M_2 + C_\epsilon.
\end{equation}
By appropriately adjusting the constant $\epsilon$, we obtain the desired pure tangential derivative estimate \eqref{pure tangential est} when either F6 holds or $G$ is uniformly concave in $p$ with respect to $u$ in $\mathcal{N}$.

Next, we shall consider the case when $\mathcal{F}$ is orthogonally invariant. For the function $v_\tau$ in \eqref{v tau}, we need to properly choose a positive constant $c_1$ in this case. We shall use the following first order approximation to the tangent vector $e_1$ at $0$,
\begin{equation}\label{xi}
   \xi = e_1 + \sum_{1\le k<n} \delta_k \nu_1(0) (x_n e_k - x_k e_n),
\end{equation}
where the $x_n$ coordinate is chosen in the direction of $\nu$ at $0$. From (2.59) in \cite{JT-oblique-I}, we have
\begin{equation}\label{tilde v1}
  \tilde v_1:= v_\xi - C_1(1+M_2)|x|^2 \le \tilde v_1(0)(1+C_1|x|^2):= f, \quad {\rm on} \ \partial\Omega.
\end{equation}
For the barrier argument, we employ the auxiliary function $v$ in \eqref{aux in Lemma 3.1} with $\alpha=0$ and $\tilde v_1$ in place of $v_1$, where $f$ is now the function defined in \eqref{tilde v1}. Assuming that the function $v$ attains its maximum over $\bar \Omega_\rho$ at some point $x^0\in \Omega_\rho$, we have $\mathcal{L}v(x^0)\le 0$. We can assume that $|D^2u(x^0)|\ge C_\sigma$ as large as we want, otherwise from $v$ constructed in \eqref{aux in Lemma 3.1}, $\phi=0$ on $\partial\Omega$ and $\phi>0$ in $\Omega_\rho$, we can obtain the upper bound for $u_{11}(0)$ and derive the desired estimate \eqref{pure tangential est}. Since the vector $\xi$ in \eqref{xi} has skew symmetric Jacobian $D\xi$, we then can reduce the calculation of $\mathcal{L}\tilde v_1$ to the argument of the proof of Lemma 2.1 in \cite{JT-oblique-I} when $\mathcal{F}$ is orthogonally invariant. By calculation, the estimate \eqref{L u11 Lf using F6} holds with $\tilde v_1$ in place of $v_1$. By choosing $c_1 = C/\epsilon$, $\kappa=8\epsilon/\sigma$ with $\epsilon$ in \eqref{3.38}, and following the argument of the proof when $\alpha=0$, we can again derive to the pure tangential estimate \eqref{pure tangential est} and complete the proof.
\end{proof}

\begin{remark}\label{Remark3.3}
If $B_p=0$ in the F6 case, we can avoid to assume $|D^2u(x^0)|$ large so that we do not need the positive lower bound of $f$ in $\mathcal{N}$. Therefore, the truncation for $\nu_1$ is not needed in this case. If $B_p=0$ in the case when $G$ is uniformly concave in $p$ with respect to $u$, we can also avoid such truncation for $\nu_1$, but only get an pure tangential estimate in the form $M_2^+(\tau) \le \epsilon M_2 + C_\epsilon(1+M_2^+(\beta))$, where $M_2^+(\beta)=\sup\limits_{\partial\Omega}u_{\beta\beta}$.
\end{remark}

From the previous estimates established in this section, we obtain the following global second derivative estimates.

\begin{Theorem}\label{Th3.2}
Let $u\in C^4(\Omega)\cap C^3(\bar \Omega)$ be an admissible solution of the boundary value problem \eqref{1.1}-\eqref{1.2} in a $C^{3,1}$ domain $\Omega\subset \mathbb{R}^n$, which is uniformly $(\Gamma, A, G)$-convex with respect to $u$. Assume that $\mathcal F$ is orthogonally invariant satisfying  F1-F5 and F6, $A\in C^2(\bar \Omega\times \mathbb{R}\times \mathbb{R}^n)$ is regular in $\bar\Omega$, $B > a_0, \in C^2(\bar \Omega\times \mathbb{R}\times \mathbb{R}^n)$ is convex in $p$, $G\in C^{2,1}(\partial\Omega\times \mathbb{R}\times\mathbb{R}^n)$ is oblique and uniformly concave in $p$, with respect to $u$, and there exists a function $\bar u\in C^2(\bar \Omega)$ which is admissible with respect to $u$. Assume also either {\rm (i)} F5\textsuperscript{+} holds, or {\rm (ii)} $\bar u$ satisfies the subsolution condition \eqref{subsolution} and $B$ is independent of $p$, and either {\rm (iii)}  $D_{pp}A = 0$ in $\partial\Omega\times\mathbb{R}\times \mathbb{R}^n$ or {\rm (iv)} $\mathcal G [u] \ge 0$ in some neighbourhood $\mathcal{N}$ of $\partial \Omega$.  Then we have the following estimate
\begin{equation}\label{2nd derivative bd for nonlinear G}
\sup_{\Omega}|D^2u| \le C,
\end{equation}
where the constant $C$ depends on $n, A, B, F, G,\beta_0, \Omega, \bar u$ and $|u|_{1;\Omega}$.
\end{Theorem}
\begin{proof}
We first observe that the global second derivative estimate \eqref{global 2nd bound regular A} in Theorem \ref{Th3.1}, the mixed tangential oblique derivative estimate \eqref{mixed nonlinear G}, the pure oblique derivative estimate \eqref{upper bound for double oblique}, the pure tangential derivative estimate \eqref{pure tangential est} in Lemma \ref{Lemma 3.1}, hold under the assumptions of Theorem \ref{Th3.2}.
Combining these boundary derivative estimates \eqref{mixed nonlinear G}, \eqref{upper bound for double oblique} and \eqref{pure tangential est}, we get an estimate
\begin{equation}
  \sup_{\partial\Omega}u_{\xi\xi} \le \epsilon M_2 + C_\epsilon,
\end{equation}
for any constant unit vector $\xi$ and constant $\epsilon>0$, where $C_\epsilon$ is a constant depending on $\epsilon, F, A, B, G, \beta_0, \Omega$ and $|u|_{1;\Omega}$. Using the concavity property F2 of $F$, we get the full boundary estimate
\begin{equation}\label{full 2nd bd on the boundary}
   \sup_{\partial\Omega}|D^2u| \le \epsilon M_2 + C_\epsilon,
\end{equation}
for any $\epsilon>0$. From \eqref{full 2nd bd on the boundary} and \eqref{global 2nd bound regular A}, the second derivative estimate \eqref{2nd derivative bd for nonlinear G} holds by taking $\epsilon$ sufficiently small.
\end{proof}

\begin{remark}\label{Remark3.4}
In Lemma \ref{Lemma 3.1} and Theorem \ref{Th3.2}, we need only assume a tangential concavity condition on $G$. Namely we call the function $G\in C^2(\partial\Omega \times \mathbb{R}\times \mathbb{R}^n)$ tangentially locally uniformly concave, (tangentially concave), in $p$ if $G_{p_kp_l}\tau_k\tau_l < 0, ( \le 0),$ in $\partial\Omega\times \mathbb{R}\times \mathbb{R}^n$ for all tangential vectors $\tau$.  If
\begin{equation}\label{tangential concavity}
G_{p_kp_l}(\cdot, u, Du)\tau_k\tau_l \le - \gamma_0 |\tau|^2, \quad {\rm on}  \ \partial\Omega,
\end{equation}
for all tangential vectors $\tau$,  $u\in C^1(\bar \Omega)$ and some constant $\gamma_0>0,  (= 0)$, then we say the function $G$ is tangentially uniformly concave, (tangentially concave) in $p$, with respect to $u$. From the proof of Lemma \ref{Lemma 3.1}, we can replace the concavity of $G$ with respect to $p$ by tangential concavity in Lemma \ref{Lemma 3.1}, for the cases where F6 holds or $\mathcal F$ is orthogonally invariant and consequently also in Theorem \ref{Th3.2}. Also replacing $G$ by
\begin{equation}
\tilde G:= 1-e^{-cG}
\end{equation}
 for a sufficiently large positive constant $c$, it follows that if $G$ is oblique and tangentially locally uniformly concave,  then
$\tilde G$ is locally uniformly concave in $p$, in the sense that
\begin{equation}
\tilde G_{p_kp_l}\xi_k\xi_l \le -\delta|\xi|^2
\end{equation}
for all vectors $\xi \in \mathbb{R}^n$, $x\in \partial \Omega$ and $|z| + |p| \le K$ for any constant $K$ where $\delta$ is a positive constant depending on $K$ and $G$. It follows then that for oblique $\mathcal G$ in Lemma \ref{Lemma 3.1}, we need only assume $G$ is tangentially locally uniformly concave in the general case. A particular example of an oblique, tangentially locally uniformly concave boundary operator $\mathcal G$, which was considered by Urbas in \cite{Urbas1995, Urbas1998}, is given by
\begin{equation}\label{another nonlinear G}
G(x,z,p)=p\cdot \nu  + G^\prime(x,z,p^T),
\end{equation}
where $p^T=p-(p\cdot \nu)\nu$ and $G^\prime$ is tangentially locally uniformly concave as a function of $p$.
\end{remark}

The case (iv) of Theorem \ref{Th3.2} can be applied to the second boundary value problem of certain augmented Hessian equations. Here we are given a $C^1$ mapping $Y$ from $\bar\Omega\times \mathbb{R}\times \mathbb{R}^n$, satisfying det$Y_p \ne 0$ and the matrix function $A$ is defined by
\begin{equation}\label{A by Y}
A(x,z,p)=-Y_p^{-1}(Y_x+Y_z\otimes p).
\end{equation}
The second boundary value problem for equation \eqref{1.1} is to prescribe the image
\begin{equation}\label{second bvp}
   Tu(\Omega)=\Omega^*,
\end{equation}
where $\Omega^*$ is another given domain in $\mathbb{R}^n$ and the mapping $T$ is defined by
\begin{equation*}
   Tu:=Y(\cdot, u,Du).
\end{equation*}
 The domain $\Omega$ is called uniformly $(\Gamma, Y)$-convex with respect to $\Omega^*$ and $u\in C^0(\partial\Omega$), if
\begin{equation*}
   K_A[\partial\Omega](x,u(x),p)+\mu \nu(x) \otimes \nu(x) \in \Gamma,
\end{equation*}
for all $x\in \partial\Omega$, $Y(x,u(x),p)\in \bar\Omega^*$, and some constant $\mu>0$. If $\Omega^* \in C^2$, and $\phi \in C^2(\mathbb{R}^n)$ is a defining function for $\Omega^*$, satisfying $\phi = 0$, $D\phi \ne 0$ on $\partial\Omega$, $\phi>0$ in
$\Omega$ and $\phi <0$ outside $\bar\Omega$, then condition \eqref{second bvp} is equivalent to
 $\mathcal G[u] = 0$ on $\partial\Omega$ and $\mathcal G[u] > 0$ in $\Omega$, where $G = \phi^* \circ Y$.
Furthermore $\Omega$ is  uniformly $(\Gamma, Y)$-convex with respect to $\Omega^*$ and $u$ if and only if $\Omega$ is uniformly $(\Gamma, A,G)$-convex with respect to $u$. When $\Gamma$ is the positive cone $K^+$, we may simply refer to $\Omega$ as uniformly $Y$-convex with respect to $\Omega^*$ and $u$, which is equivalent to the inequality
\begin{equation*}
-(\delta_i\nu_j (x)- A^k_{ij}(x,u(x),p)\nu_k(x))\tau_i\tau_j \ge \delta_0
\end{equation*}
for all $x\in\partial\Omega$, $Y(x,u(x),p) \in \bar\Omega^*$, unit tangent vectors $\tau = \tau(x)$ and some constant $\delta_0 >0$, in agreement with the definitions in \cite{Tru2008,TruWang2009}. Following \cite{Tru2008}, we call the target domain $\Omega^*$,  uniformly $Y^*$ convex with respect to $\Omega$ and $u$, if
\begin{equation*}
-(\delta_i\nu^*_j (y)- (A^*)^k_{ij}(x,u(x),p)\nu^*_k(y))\tau^*_i\tau^*_j \ge \delta^*_0
\end{equation*}
for all $y\in\partial\Omega^*$, $x\in\bar\Omega^*$, $p\in \mathbb{R}^n$ satisfying $y=Y(x,u(x),p)$, unit tangent vectors $\tau^* = \tau^*(y)$ and some constant $\delta^*_0 >0$, where $\nu^*$ denotes the unit inner normal to $\Omega^*$
and
\begin{equation*}
[(A^*)_{ij}^k]= - Y_{p}^{-1}(D_{pp}Y^k)(Y_p^{-1})^t,
\end{equation*}
for each $k=1,\cdots, n$.

We then have the following corollary of Theorem \ref{Th3.2}.
\begin{Corollary}\label{Cor 3.1}
Let $u\in C^4(\Omega)\cap C^3(\bar \Omega)$ be an admissible solution of the boundary value problem \eqref{1.1}-\eqref{second bvp} with $A$ given by \eqref{A by Y}, where $\partial\Omega, \partial\Omega^* \in C^{3,1}$ and $\Omega, \Omega^*$ are uniformly $Y$-convex, uniformly $Y^*$-convex with respect to each other and $u$. Assume that
$\Gamma = K^+$ and $\mathcal F$ is orthogonally invariant satisfying  F1-F5 and F6, $A\in C^2(\bar \Omega\times \mathbb{R}\times \mathbb{R}^n)$ is regular in $\bar\Omega$, $B > a_0, \in C^2(\bar \Omega\times \mathbb{R}\times \mathbb{R}^n)$ is convex in $p$, there exists a function $\bar u\in C^2(\bar \Omega)$ which is admissible with respect to $u$. Assume also either {\rm (i)} F5\textsuperscript{+} holds, or {\rm (ii)} $\bar u$ satisfies the subsolution condition \eqref{subsolution} and $B$ is independent of $p$. Then we have the estimate
\begin{equation}
\sup_{\Omega}|D^2u| \le C,
\end{equation}
where the constant $C$ depends on $n, A, B, F, G,\delta_0, \delta^*_0, \Omega, \Omega^*, \bar u$ and $|u|_{1;\Omega}$.
\end{Corollary}

\begin{proof}
The uniform $Y^*$-convexity of $\Omega^*$ implies that $G=\phi^* \circ Y$ is uniformly concave in $p$ with respect to $u$ when $Tu$ lies in some neighbourhood $\mathcal N^*$ of $\partial\Omega^*$, for some defining function $\phi^*\in C^2(\bar \Omega^*)$ satisfying $\phi^*=0$, $D\phi^* \neq 0$ on $\partial\Omega^*$ and $\phi^*>0$ in $\Omega^*$. Following \cite{LT2016, Tru2008} we then infer an obliqueness estimate \eqref{obliqueness} for $G$ with respect to $u$. Since $\mathcal{G}[u]\ge 0$ in $\Omega$ and $G$ is clearly uniformly concave with respect to $u$ on $\partial \Omega$, we can apply Theorem \ref{Th3.2} (iv) and lead to the conclusion of Corollary \ref{Cor 3.1}.
\end{proof}

\begin{remark}\label{Remark3.5}
For the construction of admissible functions in the optimal transportation and more general generated prescribed Jacobian cases see \cite{JT2014}. Also the existence of the admissible function $\bar u$ in Corollary \ref{Cor 3.1} may be replaced by the $Y$-boundedness of $\Omega$ with respect to $\Omega^*$ and $u$, as introduced in \cite{Tru2006}, that is the existence of a function $\phi \in C^2(\bar\Omega)$ satisfying
\begin{equation*}
D^2\phi(x) - A^k(x, u(x),p)D_k\phi > 0
\end{equation*}
for all $x\in\bar\Omega$, $Y(x,u(x),p)\in \bar\Omega^*$, where $A^k=D_{p_k}A$. With this alternative hypothesis the special case of Hessian quotients $F_{n,k}$ and $Y = Y(x,p)$ generated by a cost function is treated in \cite{vonNessi2010}.
\end{remark}

We consider now the situation where the lack of strict regularity of $A$ is offset by strong monotonicity conditions on either $G$ or $A$ with respect to $z$. Our main concern is with the semilinear $\mathcal{G}$ case.
When $\mathcal G$ is semilinear (or quasilinear), we still obtain
the pure second order oblique derivative estimate for $u_{\beta\beta}$ from Lemma 2.2 in  \cite{JT-oblique-I} in the extended form,
\begin{equation}\label{pure oblique}
\sup_{\partial\Omega}u_{\beta\beta}\le \epsilon M_2 + C_\epsilon(1+ M_2^\prime),
\end{equation}
for any $\epsilon>0$, where $M_2 = \sup\limits_\Omega |D^2 u|$, $M^\prime_2 = \sup\limits_{\partial\Omega} \sup\limits_{|\tau| = 1, \tau\cdot\nu=0} |u_{\tau\tau}|$, and $C_\epsilon$ is a constant depending on $\epsilon, F, A, B, G, \Omega, \beta_0$ and $|u|_{1;\Omega}$. Here we need to assume as before that  $\Omega$ is uniformly $(\Gamma, A, G)$-convex with respect to $u$, $F$ satisfies F1-F5 and either F5\textsuperscript{+} holds or $D_pB=0$. When F6 holds, $M_2^\prime$ does not appear at the right hand side of the estimate \eqref{pure oblique}, which reduces to the estimate \eqref{upper bound for double oblique}.

The remaining estimate we need to establish is the pure tangential derivative on the boundary with regular $A$. In order to ensure that the monotonicity conditions are at least independent of gradient estimates, analogous to our boundary convexity conditions, we will restrict the matrix $A$ to the special form:
\begin{equation}\label{A for monotone case}
 A(x,z,p)=A_0(x,z)+\tilde A(x)\cdot p,
\end{equation}
where $A_0 \in C^2(\bar \Omega \times \mathbb{R})$ and $\tilde A = (A^1,\cdots, A^n)$ where $A^k, k=1,\cdots,n, \in C^2(\bar \Omega)$ are $n\times n$ symmetric matrix functions. Correspondingly we will assume $B$ has the form \eqref{special B}
where  $B_0 \in C^2(\bar \Omega \times \mathbb{R})$ and $B_1(p)\in C^2(\mathbb{R}^n)$. Note that when $A$ has the form
\eqref{A for monotone case}, the $A$-curvature matrix $K_A[\partial\Omega]$ is independent of both $p$ and $z$ so that we may simply refer to  $(\Gamma, A, G)$-convexity with respect to $u$ as $(\Gamma, A)$-convexity.
As in Section 3.1 of \cite{JT-oblique-I}, it is convenient here  to normalise $G$ by dividing by $\beta\cdot\nu$ so that we may then assume $\beta\cdot\nu = 1$ on $\partial\Omega$, whence $\beta^\prime := \beta-\nu$ is tangential to $\partial\Omega$.

\begin{Lemma}\label{Lemma 3.2}
Let $u\in C^2(\bar \Omega)\cap C^4(\Omega)$ be an admissible solution of the boundary value problem \eqref{1.1}-\eqref{1.2} in a $C^{2,1}$ domain $\Omega \subset \mathbb{R}^n$. Assume that $F$ satisfies F1-F3 and F5, $A\in C^2(\bar \Omega\times \mathbb{R}\times\mathbb{R}^n)$, $B>a_0, \in C^2(\bar\Omega\times \mathbb{R}\times\mathbb{R}^n) $ is convex in $p$ and $G\in C^2(\partial\Omega\times \mathbb{R}\times \mathbb{R}^n)$ is oblique and semilinear and either F5\textsuperscript{+} holds or $B$ is independent of $p$. Assume one of the further conditions is satisfied:
\begin{itemize}
\item[(i):]  \eqref{A for monotone case}, \eqref{special B} hold in some neighbourhood $\mathcal{N}$ of $\partial\Omega$, $F$ satisfies F4, $\Omega$ is uniformly $(\Gamma, A)$-convex  and $ \kappa =\min\limits_{\partial\Omega}\varphi_z(\cdot,u) >0$;
\item[(ii):] \eqref{A for monotone case}, \eqref{special B} hold in $\Omega$ with $\kappa = \min\limits_{\Omega}\lambda_1 (A_z)(\cdot,u) > 0$, where $\lambda_1(A_z)$ denotes the minimum eigenvalue of the matrix $A_z$.
\end{itemize}
Then for any  tangential vector field $\tau$, $|\tau|\le 1$, we have the estimate
\begin{equation}\label{pure tangential}
M_2^+(\tau) \le (C M_2 + C^\prime)/\kappa,
\end{equation}
where $M_2^+(\tau)=\sup\limits_{\partial\Omega}u_{\tau\tau}$, the constant $C$ depends on $F,\Omega, \beta^\prime, \tilde A$ as well as $\inf\limits_{\mathcal N}\lambda_1 (A_z)(\cdot, u)$, $\inf\limits_{\mathcal N}B_z(\cdot, u)$ in case {\rm (i)} and $\min\limits_{\partial\Omega}\varphi_z(\cdot,u)$, $\inf\limits_{\Omega}B_z(\cdot,u)$ in case {\rm (ii)}, while the constant $C^\prime$ depends on $F, A, B, G, \Omega$ and $|u|_{1;\Omega}$.
\end{Lemma}

\begin{proof}
In this proof, we will use $C$ and $C^\prime$ to denote constants depending on the same quantities as in the statement of Lemma \ref{Lemma 3.2}. As usual, the constants $C$ and $C^\prime$ will change from line to line.

(i). The proof follows the proof of Lemma \ref{Lemma 3.1} for the nonlinear $\mathcal{G}$ case.
Note that we have already normalised $\mathcal{G}$ here such that $\beta\cdot \nu =1$ on $\partial\Omega$.
We suppose that the function $v_\tau=u_{\tau\tau}+\frac{1}{2}|u_\tau|^2$ attains its maximum over $\partial\Omega$ and tangential vectors $\tau$ satisfying $|\tau|\le 1$, at a point $x_0\in \partial\Omega$ and a vector $\tau=\tau_0$. Without loss of generality, we may assume $x_0=0$, $\tau_0=e_1$ and $\nu(0)=e_n$. Since $v_\tau$ at $0$ attains its maximum in the direction $e_1$, we can assume $\{u_{ij}+\frac{1}{2}u_iu_j\}_{i,j<n}$ is diagonal at $0$ by a rotation of the $e_2, \cdots, e_{n-1}$ coordinates.
For $\beta=\beta(x)$ with $\beta\cdot \nu =1$, setting $b$, $\tau$ as in \eqref{b}, (now $b=\nu_1$, $\tau=e_1-\nu_1\beta$), we have the same decomposition for $v_1$ on $\partial\Omega$ as in \eqref{decomposition} with $c_1=1$. By the boundary condition $D_\beta u = \varphi(x,u)$ and its differentiation in the direction of $\tau$, we have on $\partial\Omega$,
\begin{equation}\label{lm 3.2 0}
  \begin{array}{ll}
        \!\!&\!\!\displaystyle \nu_1(2u_{\beta\tau}+u_\beta u_\tau) \\
     =  \!\!&\!\!\displaystyle \nu_1[2 (D_z\varphi) u_\tau + 2 D_{x_\tau}\varphi - 2 Du\cdot D_\tau \beta + \varphi u_\tau] \\
   \le  \!\!&\!\!\displaystyle \nu_1[h(0)\tau(0)\cdot Du + g(0)] + (C M_2 + C^\prime)|x|^2\\
   \le  \!\!&\!\!\displaystyle \nu_1[h(0)u_1 + g(0)] + (C M_2 + C^\prime)|x|^2,
  \end{array}
\end{equation}
where $g:=2D_{x_\tau}\varphi - 2Du\cdot D_\tau \beta$ and $h:=2D_z\varphi + \varphi$. Therefore, from \eqref{decomposition} with $c_1=1$ and \eqref{lm 3.2 0}, we have
\begin{equation}\label{lm 3.2 1}
\begin{array}{ll}
  v_1  \le \!\!&\!\!\displaystyle (1-2\nu_1\beta_1 + C_0 \nu_1^2)v_{1}(0) + h(0)\nu_1u_1 + g(0)\nu_1\\
      \!\!&\!\!\displaystyle  + (C M_2 + C^\prime)|x|^2:=f, \quad {\rm on} \ \partial\Omega,
\end{array}
\end{equation}
where the constant $C_0$ depends on $\Omega$ and $\beta$, and $f(0)=v_{1}(0)$. By making $C_0$ larger if necessary, we may assume that $1-2\nu_1\beta_1 + C_0 \nu_1^2 \ge 1/2$ in $\Omega$.
We now define an auxiliary function
\begin{equation}\label{aux in Lemma 3.2}
   v:= v_1 - f - c \phi,
\end{equation}
where $v_1=u_{11}+\frac{1}{2}|u_1|^2$, $f$ and $\phi$ are functions in \eqref{lm 3.2 1} and \eqref{barrier}, $c$ is a fixed constant to be chosen later. We consider the function $v$ in the boundary neighbourhood $\mathcal{N}=\Omega_\rho=\{x\in \Omega| \ d(x)<\rho\}$ with a small positive constant $\rho$. We assume that the function $v$ attains its maximum over $\bar \Omega_\rho$ at some point $x^0\in \Omega_\rho$. Then we have $\mathcal{L}v\le 0$ at $x^0$.

 Since $F$ satisfies F4, from the uniform $(\Gamma, A)$-convexity of $\Omega$, the barrier inequality \eqref{full barrier} holds provided $|D^2u|\ge C_\sigma$ if F5\textsuperscript{+} holds, while $\mathcal{L}=L$ satisfies \eqref{barrier} if $B_p=0$. As in the proof Lemma \ref{Lemma 3.1}, from the properties of $f$ and $\phi$, and the construction of $v$, we may assume $|D^2u(x^0)|\ge C_\sigma$ as large as we want. Actually, we can assume $u_{11}(x^0)>1$ as large as we want.

We now calculate $\mathcal{L}v$ in detail. Using F2 and convexity of $B$ in $p$, and taking $\tau=e_1$, in \eqref{twice diff}, we have
\begin{equation}\label{lm 3.2 3}
  \mathcal{L}u_{11} \ge F^{ij}[\tilde D_{x_1x_1}A_{ij}+ A_{ij}^{kl}u_{1k}u_{1l}+2(\tilde D_{x_1}A_{ij}^k)u_{1k}] + \tilde D_{x_1x_1}B + 2(\tilde D_{x_1}D_{p_k}B)u_{1k},
\end{equation}
in $\Omega$. Using \eqref{A for monotone case}, \eqref{special B} in $\mathcal{N}$ and F5, we then obtain from \eqref{lm 3.2 3},
\begin{equation}\label{lm 3.2 4}
\begin{array}{ll}
  \mathcal{L}u_{11} \!\!&\!\!\displaystyle \ge F^{ij}[D_zA_{ij}u_{11} +2(\tilde D_{x_1}A_{ij}^k)u_{1k}] + (D_zB^0) u_{11}
  - C^\prime (1+\mathscr{T})\\
                                    \!\!&\!\!\displaystyle \ge - (C M_2 + C^\prime) \mathscr{T}.
\end{array}
\end{equation}
Next, by direct calculation, we have
\begin{equation}
\begin{array}{ll}
\mathcal{L}(\nu_1 u_1) \!\!&\!\!\displaystyle  = \nu_1 \mathcal{L} u_1 + u_1 \mathcal{L}\nu_1 + 2F^{ij}\nu_{1i}u_{1j} \\
                       \!\!&\!\!\displaystyle  \le C^\prime \mathscr{T} + \epsilon F^{ij}u_{1i}u_{1j} + \frac{1}{\epsilon} F^{ij}\nu_{1i}\nu_{1j},
\end{array}
\end{equation}
for any constant $\epsilon>0$, where (2.26) in \cite{JT-oblique-I} and the Cauchy's inequality are used in the inequality. Consequently, taking $\epsilon=1/|h(0)|$, we have
\begin{equation}\label{L nu1u1}
  h(0) \mathcal{L}(\nu_1 u_1) \le C^\prime \mathscr{T} +  F^{ij}u_{1i}u_{1j}.
\end{equation}
By using \eqref{L nu1u1} and further calculation, we have
\begin{equation}\label{lm 3.2 6}
  \begin{array}{rl}
    \displaystyle \mathcal{L}(\frac{1}{2}|u_1|^2 -f)= \!\!&\!\! \displaystyle F^{ij}u_{1i}u_{1j} + u_1\mathcal{L}u_1 - \mathcal{L}f \\
   \ge\!\!&\!\! \displaystyle - (CM_2 + C^\prime) \mathscr{T}.
  \end{array}
\end{equation}
From \eqref{full barrier}, \eqref{lm 3.2 4} and \eqref{lm 3.2 6}, we have
\begin{equation}\label{contradiction lm 3.2}
\begin{array}{ll}
   0 \!\!&\!\!\displaystyle \ge \mathcal{L} v(x^0) =\mathcal{L}(v_1-f-c \phi)(x^0)\\
     \!\!&\!\!\displaystyle \ge \frac{c}{2}\sigma  \mathscr{T} - (CM_2 + C^\prime)\mathscr{T}\\
     \!\!&\!\!\displaystyle \ge  \mathscr{T} >0,
\end{array}
\end{equation}
by fixing the constant $c$ such that $c> 2(CM_2 + C^\prime+1)/\sigma$.
The contradiction in \eqref{contradiction lm 3.2} implies that $v$ attains its maximum on $\partial\Omega_\rho$. We can choose $c$ large such that $v<0$ on $\Omega\cap\partial\Omega_\rho$. By \eqref{lm 3.2 1}, we have $v\le 0$ on $\partial\Omega$ and $v=0$ at $0$. Therefore, $v$ attains its maximum in $\bar\Omega_\rho$ at the point $0\in \partial\Omega$. Thus, we have $D_\beta v(0)\le 0$, which now gives
\begin{equation}\label{u 110 le CM2}
   u_{11\beta}\le -2(D_\beta \nu_1)\beta_1u_{11}+CM_2 + C^\prime, \quad {\rm at} \  0.
\end{equation}

On the other hand, by tangentially differentiating $D_\beta u = \varphi(\cdot, u)$ twice, with $\beta = \beta(x)$ on $\partial\Omega$, we have
\begin{equation}\label{twice diff bdy semilinear case}
u_{\tau\tau\beta} = (D_z \varphi)u_{\tau\tau} - 2(D_\tau \beta_k)u_{k\tau} -[u_k D_{\tau\tau}\beta_k - D_{x_\tau x_\tau}\varphi - 2u_\tau D_{zx_\tau}\varphi -u_\tau^2 D_{zz}\varphi],
\end{equation}
on $\partial\Omega$, where $D_{x_\tau}=\tau\cdot D_{x}$. At the point $0\in \partial\Omega$, taking $\tau=e_1$ in \eqref{twice diff bdy semilinear case}, recalling that $\{u_{ij}+\frac{1}{2}u_iu_j\}_{i,j<n}$ is diagonal at $0$, we have
\begin{equation}\label{u 11beta ge semilinear case}
\begin{array}{ll}
u_{11\beta} \!\!&\!\! \ge  (D_z\varphi-2D_1\beta_1) u_{11} - 2\sum\limits_{k>1}(D_1\beta_k)u_{1k}- C^\prime\\
            \!\!&\!\! \ge  (D_z\varphi-2D_1\beta_1) u_{11} - 2(D_1\beta_n)u_{1n} + \sum\limits_{1<k<n}(D_1\beta_k)u_{1}u_{k}- C^\prime \\
            \!\!&\!\! \ge  (D_z\varphi-2D_1\beta_1) u_{11} - 2(D_1\beta_n)u_{1n} - C^\prime,
\end{array}
\end{equation}
at $0$.
By expressing $e_n$ in terms of $\beta(0)$ and the tangential components, we have $e_n=-\sum\limits_{k<n}\beta_ke_k + \beta$ at $0$. Since $\{u_{ij}+\frac{1}{2}u_iu_j\}_{i,j<n}$ is diagonal at $0$, we have
\begin{equation}\label{u 1n}
  u_{1n}=-\beta_1 u_{11}+\frac{1}{2}\sum\limits_{1<k<n}\beta_k u_{1} u_{k}+u_{1\beta},
\end{equation}
at $0$.
From \eqref{u 110 le CM2}, \eqref{u 11beta ge semilinear case} and \eqref{u 1n}, we obtain
\begin{equation}\label{[]u11(0)le}
[D_z\varphi-2D_1\beta_1 + 2(D_\beta \nu_1) \beta_1+2 (D_1\beta_n)\beta_1](0)u_{11}(0) \le CM_2+C^\prime.
\end{equation}
For the desired estimate \eqref{pure oblique}, we can rewrite \eqref{[]u11(0)le} in the form
\begin{equation}\label{rewritten[]}
D_z\varphi(0,u(0)) u_{11}(0) \le CM_2+C^\prime,
\end{equation}
where the terms $[-2D_1\beta_1 + 2(D_\beta \nu_1) \beta_1+2 (D_1\beta_n)\beta_1](0)u_{11}(0)$ in \eqref{[]u11(0)le} are absorbed into $CM_2$ on the right hand side of \eqref{rewritten[]}.
The estimate \eqref{pure oblique} now follows immediately from \eqref{rewritten[]}.

(ii). As in (i), we suppose that the function $v_\tau=u_{\tau\tau}+\frac{1}{2}|u_\tau|^2$ attains its maximum over $\partial\Omega$ and tangential vectors $\tau$ satisfying $|\tau|\le 1$, at a point $x_0\in \partial\Omega$ and a vector $\tau=\tau_0$. Without loss of generality, we may assume $x_0=0$, $\tau_0=e_1$ and $\nu(0)=e_n$. We can assume that $u_{11}(0)>0$. Following the analysis in (i), the inequalities \eqref{lm 3.2 0}, \eqref{lm 3.2 1} hold on $\partial\Omega$. Then we have
\begin{equation}
v_1-f \le 0, \quad {\rm on} \ \partial\Omega, \quad{\rm and} \ v_1-f=0, \ {\rm at} \ 0\in \partial\Omega,
\end{equation}
where $v_1$ and $f$ are functions defined in \eqref{lm 3.2 1}. From \eqref{u 11beta ge semilinear case}, we have
\begin{equation}\label{u 11beta ge semilinear case'}
u_{11\beta}(0) \ge -(CM_2 + C^\prime).
\end{equation}
Using \eqref{mixed nonlinear G} and \eqref{u 11beta ge semilinear case'}, we have
\begin{equation}
D_\beta (v_1-f)(0)\ge -(CM_2 + C^\prime).
\end{equation}
Then the function
\begin{equation}
v := v_1 - f - (CM_2+C^\prime)\phi
\end{equation}
satisfies $D_\beta v (0)>0$ for properly larger constants $C$ and $C^\prime$, where $\phi\in C^2(\bar \Omega)$ is a negative defining function for $\Omega$ satisfying $\phi=0$ on $\partial\Omega$ and $D_\nu \phi=-1$ on $\partial\Omega$. Therefore, $v$ must take its maximum over $\bar \Omega$ at an interior point in $\Omega$. We assume that $v$ takes its maximum at $x^0\in \Omega$. Then we have $\mathcal{L}v \le 0$, at $x^0$.
Under our assumptions for $F, A$ and $B$, by direct calculation as in (i), we have
\begin{equation}\label{L u11 c1u12}
\begin{array}{ll}
  \mathcal{L}v \!\!&\!\!\displaystyle \ge F^{ij}(D_zA_{ij})u_{11} - (CM_2+C^\prime) \mathscr{T} \\
               \!\!&\!\!\displaystyle \ge \kappa \mathscr{T}u_{11} - (CM_2+C^\prime) \mathscr{T}.
\end{array}
\end{equation}
Combining \eqref{L u11 c1u12} with $\mathcal{L}v(x^0)\le 0$, and using F5, we have
\begin{equation}
u_{11}(x^0) \le (CM_2+C^\prime)/\kappa.
\end{equation}
By the forms of $v$, $f$ and $\phi$, we can derive the desired pure tangential estimate \eqref{pure tangential} and complete the proof for (ii).
\end{proof}

\begin{remark}\label{Remark3.6}
The explicit form of the coefficient of $u_{11}(0)$ in the estimate \eqref{[]u11(0)le} permits some refinement of the constant
$\kappa$ in case (i).  In particular if $\Omega$ is convex with minimum and maximum boundary curvatures, $\kappa_ 1$ and $\kappa _{n-1}$  respectively and $|\beta^\prime|\le \alpha\kappa_1/\kappa_{n-1}$ for some constant  $\alpha <1$, then we can take  $\kappa = \min\limits_{\partial\Omega}[\varphi_z(\cdot,u) + 2(1-\alpha)\kappa_1]$. Note that without the normalisation
 $\beta\cdot \nu=1$ on $\partial\Omega$ we should divide $\beta$ by  $\beta\cdot \nu$.
\end{remark}

Combining the global estimate \eqref{global 2nd bound regular A} or \eqref{improved 2nd global estimate}, together with the boundary estimates \eqref{mixed nonlinear G}, \eqref{pure oblique} and \eqref{pure tangential} in the semilinear $\mathcal{G}$ case, we can obtain the full second derivative estimate under strong monotonicity condition on either $G$ or $A$ with respect to $z$. For this purpose we will employ the following refinement of \eqref{pure oblique}
\begin{equation}\label{pure oblique improved}
\sup_{\partial\Omega}u_{\beta\beta}\le \epsilon M_2 + C_\epsilon M_2^\prime + C_\epsilon^\prime,
\end{equation}
for any $\epsilon>0$,  where $C_\epsilon$ is a constant depending on $\epsilon$ and $\beta$ and $C_\epsilon^\prime$ is a  constant depending on $\epsilon, F, A, B, G, \Omega$ and  $|u|_{1;\Omega}$. To get \eqref{pure oblique improved} from the proof of Lemma 2.2 in \cite{JT-oblique-I}, we have in (2.29) of \cite{JT-oblique-I}
$$F^{ij}\beta_{ik}u_{jk} = 2F^{ij} (D_i\beta_k)u_{jk} -(D_z\varphi) F^{ij}u_{ij} $$
so that by using the inequalities (1.9) and (1.10) in \cite{JT-oblique-I} to estimate $F^{ij}u_{ij}$, (see also (3.24) in
\cite{JT-oblique-I} and \eqref{homogeneity}  in this paper), the constant $C$ in (2.29) of  \cite{JT-oblique-I} need only depend
on $\beta$   provided instead the constant $C_{\epsilon_1}$  is allowed to depend on
$\epsilon_1, F, A, B, G, \Omega$ and $|u|_{1;\Omega}$. By examining the rest of the proof of Lemma 2.2 in \cite{JT-oblique-I}, we then obtain \eqref{pure oblique improved}. Moreover we can take $C_\epsilon =  \frac{C}{\epsilon}$ for some constant $C$ depending on $F, A, B, G, \Omega$ and $\beta_0$. Also to avoid dependence on $|u|_{1;\Omega}$ in global estimate \eqref{global 2nd bound regular A}, we will use the sharper form in \eqref{improved 2nd global estimate} in Remark \ref{Remark2.1}, where the coefficient of $\sup_{\partial\Omega}|D^2u|$ is $1$. From the above considerations, we now have the following full second derivative estimate for semilinear $\mathcal{G}$. As in Theorem 1.2 in \cite{JT-oblique-I}, we need to assume the cone $\Gamma$ lies strictly in a half space in the sense that $r\le {\rm trace}(r)I$ for all $r\in \Gamma$. For simplicity we also assume that $\varphi, A$ and $B$ are non-decreasing in $z$, which implies that the constants $C$ in Lemma \ref{Lemma 3.2} depend only on $F, \Omega, \beta^\prime$ and $\tilde A$.

\begin{Theorem}\label{Th3.3}
Let $u\in C^4(\Omega)\cap C^{2}(\bar \Omega)$ be an admissible solution of the boundary value problem \eqref{1.1}-\eqref{1.2} in a $C^{3,1}$ domain $\Omega\subset\mathbb{R}^n$. Assume that $F$ satisfies conditions F1-F5, $A\in C^2(\bar \Omega\times \mathbb{R}\times \mathbb{R}^n)$ satisfies \eqref{A for monotone case} with constant $\tilde A$ in $\Omega$, $B>a_0,\in C^2(\bar \Omega\times \mathbb{R})$ satisfies \eqref{special B} and is convex in $p$, $G\in C^{2}(\partial\Omega\times\mathbb{R}\times\mathbb{R}^n)$ is oblique and semilinear, $A, B$ and $\varphi$ are nondecreasing in $z$, $\bar u\in C^2(\Omega)\cap C^1(\bar \Omega)$ is admissible with respect to $u$ and $\Omega$ is uniformly $(\Gamma, A)$-convex.  Assume either F5\textsuperscript{+} holds  or $\bar u$ satisfies the subsolution condition \eqref{subsolution} and $B$ is independent of $p$. Then there exist constants $K$ depending on $F, \Omega, \beta$ and $\tilde A$ such that if either $\varphi_z(\cdot,u) \ge K$ on $\partial\Omega$ or $A_z(\cdot,u) \ge K I$  in $\Omega$,  we have the estimate \eqref{2nd derivative bd for nonlinear G}, where $C$ is a constant depending on $F, A, B, G, \Omega, \bar u$ and $|u|_{1;\Omega}$.
\end{Theorem}

\begin{proof}
As in Lemma \ref{Lemma 3.2}, we assume $\beta\cdot\nu = 1$ on $\partial\Omega$. By the assumption of the cone $\Gamma$, the quantities $M_2^\prime$ and $M_2^+$ are equivalent. From \eqref{pure tangential} in Lemma \ref{Lemma 3.2}, and taking $K\ge 1$, we have
\begin{equation}\label{M2'}
M_2^\prime \le \frac{C_0 M_2}{K} +C^\prime,
\end{equation}
with $C_0$ depending on $F,\Omega, \beta^\prime$ and  $\tilde A$. Substituting \eqref{M2'} into \eqref{pure oblique improved}, we have
\begin{equation}\label{u betabeta}
   \sup_{\partial\Omega}u_{\beta\beta} \le (\epsilon + \frac{C_1}{\epsilon K})M_2 + C_\epsilon^\prime,
\end{equation}
for any $\epsilon\in(0,1)$, for a further constant  $C_1\ge C_0$ depends on the same quantities as $C_0$.  Consequently, from
 \eqref{mixed nonlinear G}, \eqref{M2'} and \eqref{u betabeta} we have
 \begin{equation} \label{double normal}
  \begin{array}{ll}
  u_{\nu\nu} \!\!&\!\!\displaystyle = u_{\beta\beta} - 2u_{\beta\beta^\prime} + u_{\beta^\prime\beta^\prime} \\
                     \!\!&\!\!\displaystyle  \le (\epsilon + \frac{C_1}{\epsilon K})M_2 + C_\epsilon^\prime,
   \end{array}
\end{equation}
for further constants $C_1$ and $C_\epsilon^\prime$ depending on the same quantities.  Fixing  $K = C_1\epsilon^{-2}$,
 $\epsilon =  \frac{1}{4}$, and using the concavity F2 we thus obtain, from \eqref{M2'} and \eqref{double normal},
 \begin{equation}\label{boundary estimate}
\sup_{\partial\Omega}|D^2 u| \le \frac{1}{2} M_2 + C^\prime.
\end{equation}
Combining \eqref{boundary estimate} with \eqref{global 2nd bound regular A}, we obtain the estimate \eqref{2nd derivative bd for nonlinear G}.
 \end{proof}

\begin{remark}\label{Remark3.7}
When F6 is satisfied, the pure second order oblique derivative estimate in the simpler form \eqref{upper bound for double oblique} holds and the proofs of Lemmas \ref{Lemma 3.2} and Theorem \ref{Th3.3} become simpler. We can just take $v_1=u_{11}$ in the proof of Lemma \ref{Lemma 3.2} and the estimate \eqref{2nd derivative bd for nonlinear G} can be directly obtained without the assumption of the cone $\Gamma$.
\end{remark}

\begin{remark}\label{Remark3.8}
In the special case $\tilde A = 0$, which includes the standard Hessian equations, the $(\Gamma,A)$-convexity condition becomes simply $\Gamma$-convexity and the existence of an admissible function $\bar u$ is not needed in the hypotheses of Theorem \ref{Th3.3} as the function $|x|^2$ serves the purpose, see Remark \ref{Remark3.1}. In particular for a linear boundary condition,  $D_\beta u = \gamma u + \varphi_0$, we would obtain a global second derivative bound if $\gamma \ge K$ for some constant $K$ depending on $F, \Omega$ and $\beta$, thereby extending the special case of the Monge-Amp\`ere equation in \cite{Urbas1987,Wang1992} to general Hessian equations.
\end{remark}

\section{Existence Theorems}\label{Section 4}

In this section, we shall present  some existence theorems for admissible solutions to the boundary value problem \eqref{1.1}-\eqref{1.2}. In accordance with our treatment  of the second derivative estimates in Theorems \ref{Th3.2} and \ref{Th3.3}, the situations for the nonlinear $\mathcal{G}$ case and semilinear $\mathcal{G}$ case will be discussed separately.

With the second derivative estimates in Section \ref{Section 3}, in order to establish the existence results, we also need   solution  and gradient estimates. The necessary higher order derivative estimates follow from the global H\"older estimates for second derivatives in \cite{LieTru1986}  and \cite{Tru1984} and the linear Schauder theory in \cite{GTbook}. For the solution estimates, we can assume the existence of admissible subsolutions and supersolutions of the problem \eqref{1.1}-\eqref{1.2} or alternative conditions such as in Remark 4.1 in \cite{JT-oblique-I}. Various local and global gradient estimates for the oblique boundary problem are established in Section 3 in \cite{JT-oblique-I}, which we will expand here to fit the situations in this paper.

First, we consider global gradient estimate for admissible solutions of the problem \eqref{1.1}-\eqref{1.2} when $\Gamma=K^+$, $A$ satisfies a quadratic bound from below and $G$ is tangentially  concave in $p$.  For these estimates and their subsequent application to existence, it will be convenient to express the obliqueness condition in the  form,
\begin{equation}\label{uniform obliqueness}
G_p(x,z,p)\cdot \nu \ge \beta_0,
\end{equation}
for all $x\in \Omega$, $|z|\le M$, $M \in \mathbb{R}^+$, $p\in \mathbb{R}^n$, where now $\beta_0$ is a positive constant depending on $M$. Similarly we will consider the uniform tangential concavity  of $G$ in the form,
\begin{equation}\label{uniform tangential concavity}
G_{p_kp_l}(x, z, p)\tau_k\tau_l \le - \gamma_0 |\tau|^2,
\end{equation}
for all tangential vectors $\tau$, $x\in\partial \Omega$, $|z|\le M$, $M \in \mathbb{R}^+$, $p\in \mathbb{R}^n$, where  $\gamma_0$ is a positive constant depending on $M$.

\begin{Lemma}\label{Lemma 4.1}
Let $\Omega\subset \mathbb{R}^n$ be a bounded $C^2$ domain, $u\in C^2(\Omega)\cap C^1(\bar\Omega)$ satisfying
\begin{equation}\label{u lower bound}
   D^2u\ge -\mu_0(1+|Du|^2)I,\quad {\rm in} \ \Omega,
\end{equation}
and the boundary condition \eqref{1.2}, where $\mu_0$ is a non-negative constant and $G$ is oblique and tangentially uniformly concave in $p$ satisfying \eqref{uniform obliqueness} and \eqref{uniform tangential concavity}.
Then we have the estimate
\begin{equation}\label{gradient bound in Lemma 4.1}
   \sup_{\Omega}|Du|\le C,
\end{equation}
where $C$ depends on $\mu_0, G, \Omega$ and $|u|_{0;\Omega}$. If we replace \eqref{u lower bound} by the stronger condition,
\begin{equation}\label {stronger quadratic}
D_{ij} u\xi_i \xi_j \ge - \mu_0(1 + |D_\xi u|^2)
\end{equation}
for  any unit vector $\xi$, then we need only assume $G$ is tangentially concave in $p$.
\end{Lemma}
\begin{proof}
By two applications of the Taylor's expansion and \eqref {1.2}, we have
\begin{equation}\label{Taylor expansion}
0 = G(\cdot,u,0) + [G_{p_i}( \cdot,u,p^*)\nu_i] D_\nu u + G_{p_i}(\cdot,u,0)\delta_i u +  \frac{1}{2}G_{p_ip_j} (\cdot,u, p^{**}) \delta_i u\delta_j u, \quad {\rm on} \ \partial\Omega,
\end{equation}
for some fixed vector functions $p^*$ and $p^{**}$. Specifically we can take $p^* = \delta u + t^*(D_\nu u)\nu$ and
$p^{**} = t^{**}\delta u$ where $t^*$ and $t^{**}$ are scalar functions on $\partial\Omega$ satisfying $0\le t^*, t^{**} \le 1$. From \eqref{uniform obliqueness} and \eqref{uniform tangential concavity}, we then obtain
\begin{equation}
D_\nu u \ge \frac{1}{\beta_0} ( - |G(\cdot,u,0)| - |G_p(\cdot,u,0)||\delta u| + \frac{1}{2}\gamma_0|\delta u|^2) \ge -C
\end{equation}
if $D_\nu u \le 0$, by Cauchy's inequality. The gradient estimate \eqref{gradient bound in Lemma 4.1} then follows immediately from Lemma 3.2 in \cite{JTX2015}. If we only assume $G$ is tangentially concave, we obtain from \eqref{Taylor expansion},
\begin{equation}
D_\nu u + \frac{1}{\beta_0} G_{p_i}(\cdot,u,0)\delta_i u \ge - \frac{1}{\beta_0}|G(\cdot,u,0)| \ge -C
\end{equation}
so that if \eqref {stronger quadratic} holds we obtain \eqref {gradient bound in Lemma 4.1} from Theorem 2.2 in \cite{LTU1986}
as condition \eqref {stronger quadratic} implies the function $e^{\kappa u}$ is semi-convex for large $\kappa$; see also Lemma 4.1 in \cite{JTX2015}.
\end{proof}
Lemma \ref{Lemma 4.1} provides a gradient estimate for $u$ satisfying \eqref{u lower bound} and \eqref{1.2} with nonlinear and tangentially uniformly concave $G$, which is a counterpart for the Dirichlet boundary condition case in \cite{JTY2013} and the semilinear Neumann and oblique boundary condition cases in \cite{JTX2015}. For $\Gamma=K^+$, an admissible $u\in C^2(\Omega)\cap C^1(\bar \Omega)$ satisfies $M[u]\in K^+$. When we assume a quadratic bound for $A$ from below, namely,
\begin{equation}\label{A lower quadratic}
A(x,z,p)\ge -\mu_0(1+|p|^2)I,
\end{equation}
for all $x\in \Omega$, $|z|\le M$, $p\in \mathbb{R}^n$ with $M\in \mathbb{R}^+$ and some non-negative constant $\mu_0$ depending on $M$. Then the admissibility in $K^+$ and the condition \eqref{A lower quadratic} imply that $u$ satisfies the inequality \eqref{u lower bound}. Therefore, we can apply the gradient estimate in Lemma \ref{Lemma 4.1} to the admissible solutions of the problem \eqref{1.1}-\eqref{1.2} in $K^+$ for $A$ in \eqref{A lower quadratic} and tangentially uniformly concave $G$ in $p$. As in \cite{JT-oblique-I}, we can also use $A\ge O(|p|^2)I$ as $|p|\rightarrow \infty$ to simply denote the condition \eqref{A lower quadratic}. Also under the stronger condition
\begin{equation}\label{A stronger quadratic}
A_{ij}(x,z,p)\xi_i\xi_j \ge -\mu_0(1+|p\cdot\xi|^2),
\end{equation}
for any unit vector $\xi$, we need only assume $G$ is tangentially concave in $p$.

For the maximum modulus estimate of the solution, we will assume the existence of an admissible subsolution $\underline u\in C^2(\Omega)\cap C^1(\bar\Omega)$ and a supersolution $\bar u\in C^2(\Omega)\cap C^1(\bar \Omega)$ of the oblique boundary value problem \eqref{1.1}-\eqref{1.2} in the sense of
\begin{equation}\label{sub-}
F(M[\underline u]) \ge B(\cdot,\underline u, D\underline u), \ {\rm in} \ \Omega, \quad G(\cdot, \underline u, D\underline u)\ge 0, \ {\rm on} \ \partial\Omega,
\end{equation}
and
\begin{equation}\label{sup-}
F(M[\bar u]) \le B(\cdot,\bar u, D\bar u), \ {\rm in} \ \Omega, \quad G(\cdot, \bar u, D\bar u)\le 0, \ {\rm on} \ \partial\Omega,
\end{equation}
respectively. If we assume that $A, B$ and $G$ are non-decreasing with respect to $z$, with at least one of them strictly increasing, by the comparison principle, an admissible solution $u\in C^2(\Omega)\cap C^1(\bar \Omega)$ of the problem \eqref{1.1}-\eqref{1.2} satisfies $\underline u \le u \le \bar u$ in $\bar\Omega$, see Section 4.1 in \cite{JT-oblique-I}.

Combining the solution bound from the subsolution and supersolution, gradient bound in Lemma \ref{Lemma 4.1}, and the second derivative bound in Theorem \ref{Th3.3}, we can now formulate an existence theorem for admissible solutions for regular $A$ and locally uniformly concave $G$ when $\Gamma=K^+$.
\begin{Theorem}\label{Th4.1}
Assume that $\mathcal{F}$ is orthogonally invariant and satisfies F1-F4 and F6 in the positive cone $K^+$, $\Omega$ is a $C^{3,1}$ bounded domain in $\mathbb{R}^n$, $A\in C^2(\bar \Omega \times \mathbb{R}\times \mathbb{R}^n)$ is regular in $\bar \Omega$ with $D_{pp}A = 0$ in $\partial\Omega\times\mathbb{R}\times \mathbb{R}^n$, $B>a_0, \in C^2(\bar \Omega\times \mathbb{R}\times \mathbb{R}^n)$ is convex in $p$, $G\in C^{2,1}(\partial\Omega\times \mathbb{R}\times \mathbb{R}^n)$ is uniformly oblique satisfying \eqref{uniform obliqueness} and locally tangentially uniformly concave in $p$, $\underline u$ and $\bar u, \in C^2(\Omega)\cap C^1(\bar \Omega)$ are respectively an admissible subsolution and a supersolution of the oblique boundary value problem \eqref{1.1}-\eqref{1.2} in the sense of \eqref{sub-} and \eqref{sup-}, with $\bar u \le u_0$, where $u_0\in C^2(\bar\Omega)$ is  also admissible and  $\Omega$ is uniformly $A$-convex with respect to the interval $\mathcal{I}=[\underline u, \bar u]$. Assume also that, $A$, $B$ and $-G$ are non-decreasing in $z$, with at least one of them increasing, either $A$ satisfies \eqref{A lower quadratic} and $G$ is tangentially uniformly concave satisfying \eqref{uniform tangential concavity} or A satisfies \eqref{A stronger quadratic} and either {\rm (i)} F5\textsuperscript{+} holds, or {\rm (ii)} $B$ is independent of $p$ and $u_0$ is a subsolution of \eqref{1.1}. Then there exists a unique admissible solution $u\in C^{3,\alpha}(\bar \Omega)$ of the boundary value problem \eqref{1.1}-\eqref{1.2} for any $\alpha<1$.
\end{Theorem}

Since the condition F5 itself is a consequence of F2 and F4, (see \cite{JT-oblique-I}), we only assume conditions F1-F4, F6 and omit the condition F5 in the statement of Theorem \ref{Th4.1}. The admissible function $u_0$ corresponds to the function $\bar u$ in Lemma \ref{Lemma 2.1} and is not needed in the hypotheses when $A$ and $B$ are independent of $z$. With the {\it a priori} derivative estimates up to second order, Theorem \ref{Th4.1} is readily proved by the method of continuity, similarly to the proof of Theorem 4.1 in \cite{JT-oblique-I}. Here we note that Theorem \ref{Th4.1} requires that the matrix $A(x,z,p)$ is affine in $p$ when $x$ lies in $\partial\Omega$. Such a condition is not needed in the strictly regular case in \cite{JT-oblique-I}. Moreover using Lemma \ref{Lemma 4.1}, we can remove the second inequality in condition (4.3) in the hypotheses of the corresponding existence Theorem 4.2 in \cite{JT-oblique-I} if either $G$ satisfies \eqref{uniform tangential concavity} or $A$ satisfies \eqref{A stronger quadratic}. Note that we also need only assume that $G$ is tangentially concave in Theorem 4.2 in \cite{JT-oblique-I}, (as well as in Theorem 1.2 and Lemma 2.3 in \cite{JT-oblique-I} and Theorems 4.1 and 4.2 in \cite{JTX2015}).

We remark that there is a natural example for a regular matrix $A$ satisfying the assumptions of Theorem \ref{Th4.1}, obtained by taking $A(x,z,p)=\zeta(x) \bar A(x,z,p)$ for  a non-negative function $\zeta \in C^2(\bar \Omega)$ vanishing on
$\partial\Omega$ and a non-decreasing regular matrix $\bar A\in C^2(\bar \Omega \times \mathbb{R}\times \mathbb{R}^n)$. In such an example, and more generally when $ D_p A = 0$ on $\partial\Omega$, uniform $A$-convexity reduces to simply uniform convexity.

Next, we consider a global gradient estimate for admissible solutions of the problem \eqref{1.1}-\eqref{1.2} for semilinear $\mathcal{G}$, under suitable conditions on $\varphi, A$ and $B$ corresponding to the hypotheses of the second derivative  estimate in Theorem 3.3. For simplicity we will assume that $\varphi$, $A$ and $B$ are all non-decreasing  with respect to $z$ and as in Remark \eqref{Remark3.8}, $\varphi$ is affine in $z$, that is
\begin{equation}\label {linear bc}
  \varphi(\cdot, z) = \gamma z + \varphi_0,
 \end{equation}
  where $\gamma \ge 0$ and $\varphi_0$ are functions on $\partial\Omega$. Also as in Remark 3.8 we will also consider the special case when $A = A_0(x,z)$ is independent of $p$, as well as $B =B_0(x,z)$, although as we shall indicate below such restrictions are not needed in the proof of the case  when $A_z$ is large.

\begin{Lemma}\label{Lemma 4.2}
Let $u\in C^3(\Omega)\cap C^2(\bar \Omega)$ be an admissible solution
of the boundary value problem \eqref{1.1}-\eqref{1.2} for an oblique,
semilinear boundary operator $\mathcal{G}$ in a $C^{2,1}$ domain
$\Omega\subset \mathbb{R}^n$. Assume that $\mathcal{F}$ satisfies
F1-F4, $A = A_0 \in C^2(\bar \Omega\times \mathbb{R})$, $B=B_0>a_0,\in
C^2(\bar \Omega\times \mathbb{R})$, $\beta\in C^2(\partial\Omega)$,
$\varphi \in C^2(\partial\Omega\times \mathbb{R})$ is given by
\eqref{linear bc} with $A, B$ and $\varphi$ non-decreasing in $z$.
Then there exist constants $K$ depending on $F, \Omega$ and $\beta$
such that if either {\rm (i)} $A_z(\cdot,u) \ge K I$ in $\Omega$ or
{\rm (ii)} $\gamma \ge K$ on $\partial\Omega$ and $\Omega$ is
uniformly $\Gamma$-convex, we have the estimate
\begin{equation}\label{gb Lemma4.2}
\sup_\Omega|Du| \le C,
\end{equation}
where $C$ depends on $F, A, B, \Omega, \beta, \varphi$ and $|u|_{0;\Omega}$.
\end{Lemma}
\begin{proof}
The proof depends on careful examination and modification of the proof
of case (i) in Theorem 1.3 in \cite{JT-oblique-I}. First we need to
refine the differential inequality (3.9) there with
\begin{equation} \label{g}
g= |\delta u|^2 + |D_\beta u - \varphi(\cdot,u)|^2
\end{equation}
to obtain, under our assumed conditions on $\varphi, A$ and $B$,
\begin{equation} \label{Lg}
\mathcal L g \ge a_1[\mathcal E^\prime_2 + F^{ij}(D_zA_{ij}) |Du|^2] -
C_0( \gamma^2 F^{ij}u_iu_j + \mathscr{T} |Du|^2) - C^\prime (1 +
\mathscr{T}) ,
\end{equation}
where $a_1 = (1 + \beta^\prime_0)^{-2}$, $\beta^\prime_0 = \sup
|\beta^\prime|$ and $C_0$ is a constant depending on $\beta$ and
$\Omega$ while the constant $C^\prime$ depends also on $\varphi, A, B$
and $M_0 = |u|_{0;\Omega}$. Here we normalise $\beta\cdot\nu = 1$ and
write $\beta^\prime = \beta - \nu$, as in the proof of Lemma
\ref{Lemma 3.2}.

In place of (3.13) in \cite{JT-oblique-I}, we will employ auxiliary
functions of the form
\begin{equation}\label{3.13}
v:= g + (-) \kappa M^2_1 \phi + \chi u^2,
\end{equation}
where $g$ is given by \eqref{g}, $\phi\in C^2(\bar\Omega)$ is a
positive defining function for $\Omega$ satisfying $\phi=0$ on
$\partial\Omega$ and $D_\nu \phi =1$ on $\partial \Omega$,
$M_1=\sup\limits_\Omega |Du|$ and
$\kappa$ and $\chi$ are suitable positive constants to be determined.
Then we have on the boundary $\partial\Omega$,
\begin{equation}\label{Dbetav}
\begin{array}{rl}
D_\beta v \ge \!\!&\!\!\displaystyle 2\gamma |\delta u|^2 - 2 \delta_k
u [(\delta_k\beta_i) D_i u + \beta_i\nu_k(D_i\nu_l)D_lu + \beta_i
(D_i\nu_k)D_\nu u - \delta_{k} \varphi_0] \\
\!\!&\!\!\displaystyle +(-)\kappa M^2_1 + 2 \chi u \varphi \\
\ge \!\!&\!\!\displaystyle (\gamma- C_0) |\delta u|^2  + (-) \kappa
M^2_1 - C^\prime (1+\chi),
\end{array}
\end{equation}
while in $\Omega$, we have from \eqref{Lg}, taking $\chi = C_0(\sup|\gamma|)^2$,
\begin{equation}\label {Lv}
\mathcal L v \ge F^{ij}(D_zA_{ij}) |Du|^2 +(-) \kappa M^2_1
\mathcal{L}\phi - C_0 \mathscr{T} |Du|^2  - C^\prime (1 +
\mathscr{T}).
\end{equation}
Now it is convenient to separate cases (i) and (ii). For case (i) we
will use the ``$+$'' sign in \eqref{3.13} and choose $\kappa$ large
enough so that $D_\beta v > 0$ on $\partial\Omega$ so that the maximum
of $v$ must occur at an interior point of $\Omega$, whence the
estimate \eqref{gb Lemma4.2} follows from \eqref{Lv}.

For case (ii) we  take the ``$-$'' sign in \eqref{3.13} and use
the uniform $\Gamma$-convexity of $\Omega$ to choose $\phi$ so that
\begin{equation}\label{global barrier}
\mathcal{L}\phi \le -\sigma\mathscr T ,\quad {\rm in} \ \Omega,
\end{equation}
for a positive constant $\sigma$. Such a function $\phi =
d-d^2/2\rho$, is already used in the proof of Lemma \ref{Lemma 3.1}
where
\eqref{global barrier} is satisfied in a neighbourhood $\Omega_\rho$
of $\partial\Omega$. In our situation here the constants $\rho$ and
$\sigma$ depend only on $\Omega$ and $\Gamma$. To extend $\phi$ to all of $\Omega$, we may fix a point $x_0\in
\Omega$ and define the truncated function $\phi_h =m_h\{\phi,\eta\}$ in
$\Omega_\rho$, where $m_h$ is the mollification of the min-function of two variables,
\begin{equation}
\eta(x) = \frac{(2D^2-|x-x_0|^2)\rho}{8D^2 }
\end{equation}
and $D$ is the diameter of $\Omega$. We then replace $\phi$ itself in
$\Omega_\rho$ by  $\phi_h$ for $h$ sufficiently small, say $h<
\rho/16$ , with $\phi = \eta$ outside $\Omega_\rho$. Then we obtain
\eqref{global barrier} with $\sigma$ replaced by min$\{\sigma,
\rho/4D^2\}$. The estimate \eqref{gb Lemma4.2} follows from
\eqref{Dbetav} by choosing $\kappa={2C_0}/{\sigma}$ in \eqref{Lv} so
that  $\mathcal L v > 0$ for $M_1$ sufficiently large and the maximum
of $v$ is taken on $\partial\Omega$.
\end{proof}

\begin{remark}\label{Remark4.1}
Similarly to Remark \ref{Remark3.6} we can use \eqref{Dbetav} to
replace $\gamma$ in case (ii) of Lemma \ref{Lemma 4.2} by
$\gamma + \kappa_1$ where $\kappa_1$ is the minimum curvature of
$\partial\Omega$. The estimate \eqref{gb Lemma4.2}
would still hold even if $\gamma$ were negative on part of
$\partial\Omega$ provided the curvature was sufficiently large on such
a region.
\end{remark}

\begin{remark}\label{Remark4.2}
In case (i) of Lemma \ref{Lemma 4.2}, we remark that $A$ and $B$ can take more general forms
which have  dependence on $p$, in particular \eqref{special A}, \eqref{special B} with
 $D_pA_1, D_pB_1 \in L^\infty(\mathbb{R}^n)$; see also Remark 3.5 in  \cite{JT-oblique-I}.
\end{remark}

Combining the gradient estimate in Lemma \ref{Lemma 4.2} and the
second derivative estimate in Theorem \ref{Th3.3}, we can establish
an existence theorem for semilinear oblique boundary value problem
\eqref{1.1}-\eqref{1.2} under strong monotonicity conditions for $G$
or $A$ with respect to $z$.
\begin{Theorem}\label{Th4.2}
Assume that $F$ satisfies conditions F1-F4, $\Omega\in C^{3,1}$ is
uniformly $\Gamma$-convex, $A = A_0 \in C^2(\bar \Omega\times
\mathbb{R})$, $B = B_0 >a_0,\in C^2(\bar \Omega\times \mathbb{R})$, $\beta\in
C^2(\partial\Omega)$, $\varphi \in C^2(\partial\Omega\times
\mathbb{R})$ is given by \eqref{linear bc} with $A, B$ and $\varphi$
non-decreasing in $z$. Then there exist constants $K$ depending on $F,
\Omega$ and  $\beta$  such that if either  $\gamma \ge K$ on
$\partial\Omega$ or  $A_z \ge K I$  in $\Omega\times\mathbb{R} $,
there exists a unique
admissible solution $u \in C^{3,\alpha}(\bar \Omega)$ of the boundary
value problem \eqref{1.1}-\eqref{1.2} for any $\alpha<1$.
\end{Theorem}

\begin{proof}
Since we are not assuming the existence of sub and
supersolutions, as for example in Theorem \ref{Th4.1}, we need {\it a priori}
solution bounds in order to apply the method of continuity. These
follow from standard arguments when $A$ and $B$ are independent of $p$
and we need only assume either $A_z$ or $\gamma$ is positive, that is
the constant $K$ need only be assumed positive. Upper bounds follow immediately from the admissibility of $u$ and the property of the cone $\Gamma\subset\Gamma_1$, which implies the simple Laplacian subsolution inequality,
\begin{equation}
\Delta u \ge A_{ii}(\cdot,u)\ge c_0 u + A_{ii}(\cdot,0),
\end{equation}
whenever $u\ge 0$, where $c_0 = \inf (D_zA_{ii})$.
To get a lower bound in the case when $A_z\ge KI$, we consider an
auxiliary function
\begin{equation}
v =\kappa \phi - u,
\end{equation}
where $\kappa$ is a positive constant and as usual $\phi\in
C^2(\bar\Omega)$ is a positive defining function for $\Omega$
satisfying $\phi=0$ on $\partial\Omega$ and $D_\nu \phi =1$ on
$\partial \Omega$. By taking $\kappa$ sufficiently large, we have
$D_\beta v >0$ on $\partial\Omega$ whenever $v\ge 0$ so that $v$ must take a positive maximum
in $\Omega$. Since
\begin{equation}
\begin{array}{ll}
 Lv \!\!&\!\!\displaystyle = \kappa L\phi -F^{ij}w_{ij} - F^{ij}A_{ij}\\
  \!&\!\!\displaystyle \ge -K\mathscr T u - C(1+\kappa) (1+\mathscr T),
\end{array}
\end{equation}
using \eqref{homogeneity}, we obtain a lower bound for $u$ at a maximum point of $v$ and hence in
$\Omega$. In the case $\gamma \ge K$, we replace $v$ by
\begin{equation}
v = \kappa \eta - u
\end{equation}
where $\eta(x) = |x|^2$ and choose $\kappa$ large enough so that $v$
takes a maximum on $\partial \Omega$. In both cases we obtain a lower bound for $u$.

Then we conclude the following maximum modulus estimate for $u$,
\begin{equation}\label{solution bound}
\sup_\Omega|u| \le C,
\end{equation}
with constant $C$ depending on $F, A, B, \beta, \gamma, \varphi_0$ and $\Omega$.

Alternatively, using condition F4, we can show functions of the form $+(-)( c_1\phi -c_0)$ are respectively admissible subsolutions and supersolutions of  \eqref{1.1}-\eqref{1.2} for sufficiently large constants $c_0$ and $ c_1$ in the case where $A_z\ge K I$, while functions of the form $+(-)( c_1\eta -c_0)$ are respectively admissible subsolutions and supersolutions of  \eqref{1.1}-\eqref{1.2} in the case $\gamma\ge K$.

Using the {\it a priori} solution estimate \eqref{solution bound}, gradient estimate \eqref{gb Lemma4.2} in Lemma \ref{Lemma 4.2}, and the second derivative estimate in Theorem \ref{Th3.3}, we obtain the existence of admissible solution $u \in C^{3,\alpha}(\bar \Omega)$ of the boundary value problem \eqref{1.1}-\eqref{1.2} for any $\alpha<1$ using the method of continuity in \cite{GTbook}. The uniqueness follows from the comparison principle.
\end{proof}

For the same reason as in Theorem \ref{Th4.1}, we omit the condition F5 for $F$ in Theorem \ref{Th4.2}. When $\gamma \ge K$ on $\partial\Omega$, Theorem \ref{Th4.2} embraces the standard Hessian equations as special cases. When $A_z \ge K I$  in $\Omega\times\mathbb{R}$, Theorem \ref{Th4.2} permits oblique boundary condition of the form $D_\beta u = \varphi_0(x)$ on $\partial\Omega$. When $\Gamma=K^+$, Theorem \ref{Th4.2} applies to the uniformly convex domain $\Omega$ in the usual sense. Note also that we do not require the orthogonal invariance of $\mathcal F$ in Theorem \ref{Th4.2}, nor in Theorem \ref{Th4.1} when $A$ and $B$ are independent of $p$.




\end{document}